\documentclass[11pt]{article}
\usepackage[a4paper]{geometry}
\usepackage[latin1]{inputenc}
\usepackage{amsmath}
\usepackage{amssymb}
\usepackage{amsthm}
\usepackage{enumerate}
\usepackage{graphicx}
\usepackage{color}
\usepackage{cite}
\usepackage{dsfont} %%pour l'indicatrice

\numberwithin{equation}{section}
\newcommand{\one}{\mathds{1}}

\newcommand{\N}{\mathbb{N}}
\newcommand{\Z}{\mathbb{Z}}

\newcommand{\F}{\mathcal{F}}
\newcommand{\Ch}{\mathcal{T}}

\newcommand{\R}{\mathbb{R}}
\newcommand{\PR}{\mathbb{P}}
\newcommand{\ESP}{\mathbb{E}}

\newcommand{\sas}{\mathcal{S}\alpha\mathcal{S}}

\newcommand{\Za}[1]{\mathrm{Z}_{\alpha} \left( {#1} \right) }

\newcommand{\La}{\mathrm{L}^{\alpha}}

\newtheorem{Theo}{Theorem}[section]
\newtheorem{Lem}{Lemma}[section]
\newtheorem{Prop}{Proposition}[section]
\newtheorem{Cor}{Corollary}[section]
\newtheorem{Def}{Definition}[section]
\newtheorem{Rem}{Remark}[section]
\def\o{\omega}
\def\O{\Omega}

\def\al{\alpha}
\def\be{\beta}
\def\ga{\gamma}
\def\la{\lambda}
\def\Lam{\Lambda}
\def\th{\theta}
\def\de{\delta}
\def\dom{\mathcal{O}}
\def\kap{\kappa}

\def\go{\mathbb{G}}

\def\ta{\mathbb{L}} 
\def\ups{\Upsilon}

\def\wt{\widetilde}
\def\wh{\widehat}
\def\I{\mathcal{I}}
\def\Ib{\mathbb{I}}

\def\jnr{\mathcal{J}_{N,r}}
\newcommand{\supp}{\mathrm{supp}}

\title{Uniformly and strongly consistent estimation for the Hurst function of a Linear Multifractional Stable Motion}

\author{Antoine Ayache 
\footnote{Corresponding author}\\
UMR CNRS 8524 Laboratoire Paul Painlev\'e\\ 
Universit\'e Lille 1, B\^atiment M2 \\ 
59655 Villeneuve d'Ascq Cedex, France\\
E-mail: \texttt{Antoine.Ayache@math.univ-lille1.fr}\\
\ 
\and
Julien Hamonier \\
 EA 2694, Laboratoire de Biomath\'ematiques\\
 Universit\'e Lille 2\\ 
 3, Rue du Professeur Laguesse, BP 83, 59006 Lille Cedex, France.\\
E-mail: \texttt{julien.hamonier@gmail.com}
 }

\begin{document}

\maketitle

\begin{abstract}
Since the middle of the 90's, multifractional processes have been introduced for overcoming some limitations of the classical Fractional Brownian Motion model. In their context, the Hurst parameter becomes a H\"older continuous function $H(\cdot)$ of the time variable~$t$. Linear Multifractional Stable Motion (LMSM) is the most known one of them with heavy-tailed distributions. Generally speaking, global and local sample path roughness of a multifractional process are determined by values of its parameter $H(\cdot)$; therefore, since about two decades, several authors have been interested in their statistical estimation, starting from discrete variations of the process. Because of complex dependence structures of variations, in order to show consistency of estimators one has to face challenging problems.

The main goal of our article is to introduce, in the setting of the symmetric $\al$-stable non-anticipative moving average LMSM, where $\al\in (1,2)$, a new strategy for dealing with such kind of problems. It can also be useful in other contexts. In contrast with previously developed strategies, this new one does not require to look for sharp estimates of covariances related to functionals of variations. Roughly speaking, it consists of expressing variations in such a way that they become independent random variables up to negligible remainders. Thanks to it, we obtain, an almost surely and $L^p(\O)$, $p\in (0,4]$, consistent estimator of the whole function $H(\cdot)$, which converges, uniformly in $t$, and even for some H\"older norms. Also, we obtain estimates for the rates of convergence. Such kind of strong consistency results in uniform and H\"older norms are rather unusual in the literature on statistical estimation of functions.
\end{abstract}

\medskip

      {\it Running head}: Uniformly and strongly consistent estimation for the Hurst function  \\

      {\it AMS Subject Classification (MSC2010 database)}: 60G22, 60G52, 62G05.\\


\section{Introduction}
\label{sec:Introestim}
Fractional Brownian Motion (FBM) (see e.g. \cite{mandelbrot1968fractional,EmMa,SamTaq}) is a quite classical random model for real-life fractal signals. Although this model offers the advantage of simplicity, it lacks flexibility and thus does not always fit with reality. One of the main reasons for limitations of FBM model is that local fractal properties of its sample paths are not really allowed to evolve over time, in other words roughness remains almost the same all along sample paths. This drawback is mainly due to the constancy in time of $H$ the Hurst parameter governing FBM. In order to overcome it, various stochastic processes belonging to the so called multifractional class have been introduced and studied since the middle of the 90's (see e.g. \cite{ayache2000generalized,ayache2002generalized,Bi,BiP,BiPP,roux1997elliptic,ayache2005multifractional,ayache2007wavelet,ayache2011multiparameter,dozzi2011,falconer2002tangent,falconer2003local,falconer2009localizable,
falconer2009multifractional,HLS12,lacaux04,ayache2011prostate,LRMT11,meerschaert2008local,peltier1995multifractional,stoev2004stochastic
,stoev2005path,stoev2006rich,Sur}). Roughly speaking, the main idea behind this new class of processes is that Hurst parameter becomes time changing, in other words a function $H(\cdot)$ depending on the time variable~$t$. In order to make statistical inference related to this functional parameter, one has to face challenging problems due to complex dependence structures of multifractional processes. Usually,  in the statistical literature on this topic, it is assumed that the available observations consist in a realization of a given multifractional process over a regular grid; estimators of $H(t_0)$, the value of the Hurst function at some arbitrary fixed time $t_0$, are built through discrete variations of observations. Then for showing their consistency, the strategies which have been developed so far, consist in looking for sharp estimates of covariances related to functionals of discrete variations (see e.g. \cite{ayache2004identification,benassi1998identifying,benassi2000identification,coeurjolly2005identification,coeurjolly2006erratum,bardet2010nonparametric,leguevel2010,hamonier2015}).
The main goal of our article, is to introduce a new strategy which, roughly speaking, consists in expressing discrete variations in such a way that they become independent random variables up to negligible remainders. 

We focus on a typical multifractional process: the symmetric $\al$-stable ($\sas$) non-anticipative moving average Linear Multifractional Stable Motion (LMSM) denoted by $\{Y(t)\,:\,t\in [0,1]\}$, which was introduced by Stoev and Taqqu in \cite{stoev2004stochastic}. Notice that, when $\al=2$, then, up to a deterministic function, the process $\{Y(t)\,:\,t\in [0,1]\}$ reduces to the Multifractional Brownian Motion (MBM) introduced in \cite{peltier1995multifractional}; we mention in passing that the latter centered Gaussian process and the closely connected one introduced in \cite{roux1997elliptic} are the two most common multifractional processes. We always assume that $\al$ belongs to the open interval $(1,2)$, no further a priori knowledge on $\al$ is required; notice that our results can easily be transposed to the Gaussian case $\al=2$.

The LMSM $\{Y(t)\,:\,t\in [0,1]\}$ will be soon defined through a stochastic integral with respect to $\Za{ds}$, an arbitrary real-valued independently scattered $\sas$ random measure on $\R$ with Lebesgue measure as its control measure; the underlying probability space is denoted by $(\O,\F,\PR)$. For later purposes, it is useful to make some brief recalls concerning this integral; the reader is referred to Chapter 3 in the book \cite{SamTaq}, for a detailed presentation of it. Generally speaking, the  stochastic integral with respect to $\Za{ds}$ is defined on  
 the Lebesgue space $L^{\al}(\R)$; the integral of an arbitrary deterministic real-valued function $f\in L^{\al}(\R) $ is denoted by $\int_{\R} f(s)\Za{ds}$, in what follows we set $\Ib(f)=\int_{\R} f(s)\Za{ds}$. Let us mention that $\Ib(f)$ is a real-valued $\sas$ random variable on $(\O,\F,\PR)$; its characteristic function is given by $\exp(-|\sigma\, \xi|^\al)$, for all $\xi\in\R$. The scale parameter $\sigma$ plays a role rather similar to that of the standard deviation of a centered real-valued Gaussian random variable. Often, instead of $\sigma$, one prefers the notation $\|\Ib(f)\|_\al$ for the scale parameter, since it satisfies
\begin{equation}
\label{eqa:scparamS}
\|\Ib(f)\|_{\al}=\Big(\int_{\R}|f(s)|^\al \,ds\Big)^{1/\al}=\|f\|_{L^\al(\R)}.
\end{equation}
One of the main differences between $\Ib(f)$ and a Gaussian random variable is that, for any positive real number $p$, the absolute moment of order $p$ of $\Ib(f)$ is infinite, unless $p\in (0,\al)$; in the latter case (see Proposition 1.2.17 in \cite{SamTaq}), one has
\begin{equation}
\label{eqa:equivmoS}
\ESP(|\Ib(f)|^{p})=c(p)\|\Ib(f)\|_{\al}^{p},
\end{equation}
where the finite positive constant $c(p)$ denotes the absolute moment of order $p$ of an arbitrary real-valued $\sas$ random variable with a scale parameter equals to $1$. Before finishing our recalls, let us emphasize that the independently scattered property of $\Za{ds}$ will play a crucial role in our article; it means that:  for each positive integer $n$ and all functions $f_1,\ldots, f_n$ belonging to $L^\al(\R)$, the coordinates of the $\sas$ random vector $(\Ib(f_1),\ldots, \Ib(f_n))$ are independent random variables as soon as the supports of $f_1,\ldots, f_n$ are disjoint up to Lebesgue-negligible sets.

Let us now turn to the definition of the LMSM $\{Y(t)\,:\,t\in [0,1]\}$. First note that the associated functional Hurst parameter $H(\cdot)$ is assumed to be defined on the compact interval $[0,1]$ and with values in an arbitrary fixed compact interval $[\underline{H},\overline{H}]$ included in the open interval $(1/\al,1)$. We impose to it in addition to satisfy, for some constant $c$, the following H\"older condition: 
\begin{equation}\label{m:cond:holder}
\forall\,\, (t_1,t_2)\in [0,1]^2, \, \,\, |H(t_1)-H(t_2)| \le c |t_1-t_2|^{\rho_H},
\end{equation}
where the order of H\"olderianity $\rho_H$ does not depend on $(t_1,t_2)$ and is such that
\begin{equation}\label{m:cond:pholder}
1\ge \rho_H >\overline{H}=\sup_{t\in [0,1]} H(t).
\end{equation}
Notice that such a H\"olderianity assumption on functional Hurst parameter is very standard in the literature on multifractional processes. 

We are now in position to precisely define the $\sas$ process $\{Y(t)\,:\,t\in [0,1]\}$.
\begin{Def}
\label{defj:lmsm}
The LMSM $\{Y(t)\,:\,t\in [0,1]\}$ with Hurst functional parameter $H(\cdot)$, is defined, for each $t\in [0,1]$, as,
\begin{equation}\label{m:def:lmsm}
Y(t)=X(t,H(t)),
\end{equation}
where $\{X(u,v):(u,v)\in [0,1]\times (1/\al,1)\}$ is the real-valued $\sas$ random field on the probability space $(\O,\F,\PR)$, such that, for every $(u,v)\in [0,1]\times (1/\al,1)$,
\begin{equation}
\label{h:def:champX}
X(u,v)=\int_{\R}\Big\{ (u-s)_+^{v-1/\alpha} - (-s)_+^{v-1/\alpha} \Big\} \Za{ds}.
\end{equation}
Recall that for all $(x,\kappa)\in\R^2$, 
\begin{equation}
\label{h:eq:pospart}
(x)_+^{\kappa}=x^{\kappa}\mbox{ if $x>0$ and}\,\,(x)_+^{\kappa}=0\mbox{ else.}
\end{equation}
\end{Def}

\begin{Rem}
\label{rem:lfsm}
When the Hurst function is a constant, then LMSM reduces to Linear Fractional Stable Motion (LFSM). Let us point out that LFSM is a quite 
classical self-similar stable stochastic process with stationary increments, descriptions of its main properties can be found in \cite{kono1991holder,Tak89,EmMa,SamTaq} for instance. 
\end{Rem}

\begin{Rem}
\label{rem:contfX}
The field $\{X(u,v):(u,v)\in [0,1]\times (1/\al,1)\}$ is always identified with its version with continuous sample paths constructed in \cite{hamonier2012lmsm} through random wavelet series. Thus, in view of (\ref{m:def:lmsm}) and (\ref{m:cond:holder}), the LMSM $\{Y(t)\,:\,t\in [0,1]\}$ has continuous sample paths as well. 
Let us mention that in addition to its continuity, the field $\{X(u,v):(u,v)\in [0,1]\times (1/\al,1)\}$ satisfies several other nice properties; a quite useful one (see Corollary~3.1 in \cite{hamonier2012lmsm})  among them, is that its sample paths 
are Lipschitz functions with respect to $v\in [\underline{H},\overline{H}]$ uniformly in $u\in [0,1]$, that is one has~\footnote{Relation (\ref{eqa:uniflipX}) is satisfied almost surely. Yet, throughout the present article, for sake of simplicity, we assume, without loss of generality, that (\ref{eqa:uniflipX}) holds on the whole probability space~$\O$.}
\begin{equation}
\label{eqa:uniflipX}
A=\sup\left\{\frac{\big|X(u,v_1)-X(u,v_2)\big|}{|v_1-v_2|}\,:\, u\in [0,1]\mbox{ and } (v_1,v_2)\in[\underline{H},\overline{H}]^2 \right\}<+\infty. 
\end{equation}
Notice that Theorem~10.5.1 on page 471 in \cite{SamTaq} allows to obtain, for some constant $c>0$ and all positive real number $z$, the following control on the tail of the distribution of $A$:
\begin{equation}
\label{eq:taildistA} 
\PR(A>z)\le c \, z^{-\al};
\end{equation}
therefore, one has $\ESP (A^q)<+\infty$, for each real number $q\in [0,\al)$.
\end{Rem}

\begin{Rem}
\label{rem:nuance}
It is possible to define LMSM through (\ref{h:def:champX}) and (\ref{m:def:lmsm}), even when $\al$ is allowed to belong to the whole interval $(0,2]$ and one has $H(t)\in (0,1)\cap(0, 1/\al]$ for some $t$'s (see \cite{stoev2004stochastic}). However, sample paths properties of such a LMSM are different from the one we consider in our article. Indeed, as soon as there exists a $t_0$ such that $H(t_0)<1/\al$, then sample paths of LMSM become discontinuous unbounded functions in any arbitrary small neighborhood of $t_0$ (see \cite{stoev2005path}). On the other hand, when $\underline{H}=\inf_{t\in [0,1]} H(t)=1/\al$, then it is not clear that LMSM has a version with continuous sample paths and that the useful property (\ref{eqa:uniflipX}) remains valid.
\end{Rem}

%\begin{Prop}[see \cite{hamonier2012lmsm}]
%\label{prop:partialXpath}
%For each fixed $\o\in\O$ and $u\in\R$, the function $v\mapsto X(u,v,\o)$ is infinitely differientiable on $(1/\al,1)$. Moreover, for all fixed $q\in\Z_+$, 
%the stochastic field $\{(\partial_v^q X)(u,v):(u,v)\in \R\times (1/\al,1)\}$ has continuous paths.
%\end{Prop}

\section{The main results: their statements and some insights on their proofs}
\label{sec:statements}
The first issue, we will deal with, is the statistical estimation of the quantity
\begin{equation}
\label{h:def1:HI}
\underline{H}(I)=\min_{t\in I} H(t),
\end{equation}
where $I\subseteq [0,1]$ denotes an arbitrary fixed compact interval with non-empty interior. 
It is important to construct an almost surely and $L^p(\O)$ consistent estimator of $\underline{H}(I)$, at least for the following two reasons.
\begin{enumerate}
\item This is for us a fundamental step in getting an almost surely and $L^p(\O)$, $p\in (0,4]$, consistent estimator of the whole function $H(\cdot)$, which converges, with respect to $t\in [0,1]$, in the uniform norm and even in some H\"older norms (see Definition~\ref{def:main2}, Theorems~\ref{thm:main2bis} and~\ref{thm:main2}, and their Corollaries~\ref{cor:cmain2} and~\ref{cor:cmain2}).
\item The estimation of $\underline{H}(I)$ is interesting in its own right. Indeed, the method used in \cite{ayache2009linear} in order to almost surely determine Hausdorff dimension of a graph of a Linear Fractional Stable Sheet, can be transposed to LMSM, and allows to show that $2-\underline{H}(I)$ is, with probability~$1$, the Hausdorff dimension of its graph on~$I$.
\end{enumerate}

Before defining the statistical estimator of $\underline{H}(I)$, we mention that it is reminiscent of more or less classical ones, previously introduced in somehow less general contexts; as for instance, the Gaussian multifractional setting (see e.g. \cite{ayache2004identification,ayache2011prostate,benassi1998identifying,benassi2000identification,coeurjolly2005identification,coeurjolly2006erratum,bardet2010nonparametric}), and the setting of the Linear Fractional Stable Motion with a constant Hurst parameter (see \cite{SPT2002}). Our main concern is not really to construct a new estimator, but to derive strong consistency results with a control, almost surely and in $L^p(\O)$, $p\in (0,4]$, on the rate of convergence. As far as we know, such kind of results on the estimation of a Hurst function, or even a constant Hurst parameter, have not yet been obtained in a context of heavy-tailed distributions.

Let us now turn to the precise definition of the estimator of $\underline{H}(I)$. To this end, we need to introduce some notations. For any fixed real number $\be\in (0,1/4]$, one denotes by 
$\Ch_\be$ the function from $[0,+\infty)$ into $[1,2^\be]$ such that, for all $x\in [0,+\infty)$,
\begin{equation}
\label{eq:troncC}
\Ch_\be (x)=\min\Big\{2^\be,\max\big\{x,2^{\be/2}\big\}\Big\}.
\end{equation}
It can easily be seen that $\Ch_\be$ is a Lipschitz function:
\begin{equation}
\label{eq:LiptroncC}
\forall\, (x_1,x_2)\in [0,+\infty)^2,\quad \big |\Ch_\be (x_1)-\Ch_\be (x_2)\big |\le |x_1-x_2|.
\end{equation}
Throughout the article the integer $L\ge 2$ is arbitrary and fixed; moreover the coefficients $a_0,a_1,\ldots, a_L$ are defined, for each $l\in\{0,\ldots,L\}$, as $$ a_l=(-1)^{L-l} \binom{L}{l}=(-1)^{L-l} \frac{L!}{l!\,(L-l)!}.$$ 
This finite sequence of real numbers $(a_l)_{0\le l\le L}$ will play a role rather similar to that of a filter in signal processing. Observe that it has exactly $L$ vanishing first moments, namely, for all $q\in\{0,\ldots, L-1\}$, one has
\begin{equation}
\label{eq1:moma}
\sum_{l=0}^L l^q a_l=0 \quad\mbox{(with the convention $0^0=1$), and}\quad\sum_{l=0}^L l^L a_l\ne 0.
\end{equation}
The estimator for $\underline{H}(I)$ is built from a filtered version, through $(a_l)_{0\le l\le L}$, of the discrete realization 
$
\big\{Y(k/N) : k\in\{0,\ldots,N\}\big\}
$
of $Y$ the LMSM; the integer $N\ge L$ being large. More precisely, it is built from the discrete variations  $\{d_{N,k}:k\in\{0,\ldots,N-L\}\big\}$ defined, for all $k\in\{0,\ldots, N-L\}$, as
\begin{equation}\label{m:djk}
d_{N,k}=\sum_{l=0}^L a_l Y\big((k+l)/N\big) =\sum_{l=0}^L a_l X\big((k+l)/N,H((k+l)/N)\big),
\end{equation}
where the last equality results from (\ref{m:def:lmsm}). Since $\underline{H}(I)$ is only connected with the restriction to $I$ of $Y$, it seems natural to focus on the variations $d_{N,k}$'s such that $k/N\in I$; therefore we consider, for each fixed $N\ge (L+1)\la (I)^{-1}$, the set of 
indices
\begin{equation}
\label{h:def:nuI}
\nu_N (I)=\big\{k\in\{0,\ldots, N-L\} :k/N\in I\big\}.
\end{equation}
The cardinality of $\nu_N (I)$ is denoted by $|\nu_N (I)|$. Notice that $|\nu_N (I)|$ does not really depends on the position of $I$, but mainly on $\la (I)$, the Lebesgue measure (i.e. the  length) of this interval. More precisely, one has
\begin{equation}
\label{card-june:nuI}
N\la (I)-L-1< |\nu_N (I)|\le N\la (I)+1;
\end{equation}
thus, as soon as $N\ge 2(L+1)\la (I)^{-1}$, one gets that
\begin{equation}
\label{card:nuI}
N\la (I)/2<|\nu_N (I)|\le 7N\la (I)/6.
\end{equation}
At last, one denotes by $V_N^\be(I)$, the empirical mean, of order $\beta$, defined as, 
\begin{equation}\label{h:def:VJ}
V_N^\be (I)=|\nu_N (I)|^{-1} \sum_{k\in \nu_N (I)} |d_{N,k}|^{\beta}.
\end{equation}
The following theorem is the first main result of the article. It provides, for each compact interval $I\subseteq [0,1]$ with non-empty interior, an estimator of  $\underline{H}(I)=\min_{t\in I} H(t)$ converging in all the spaces $L^p (\O)$, with $p\in (0,4]$. Also, it provides, independently on the position of $I$, an estimate of the rate of convergence in the $L^p (\O)$ metric.
\begin{Theo}
\label{thm:main1bis}
Assume that the Hurst function $H(\cdot)$ satisfies a H\"older condition as in (\ref{m:cond:holder}) and (\ref{m:cond:pholder}). Also assume that the real number $\be\in (0,1/4]$ and the integer $L\ge 2$ are arbitrary. Let $I\subseteq [0,1]$ be an arbitrary fixed compact interval with non-empty interior, and let $N\ge 3$ be an arbitrary integer such that 
\begin{equation}
\label{eq1:encadEVT2}
\la(I)\ge 4(L+1)N^{-1}(\log N)^2. 
\end{equation}
One sets 
\begin{equation}
\label{eq1:main1}
\wh{H}_N^\be (I)=\be^{-1}\log_2 \Bigg (\Ch_\be\bigg(\frac{V_N^\be (I)}{V_{2N}^\be (I)}\bigg)\Bigg),
\end{equation}
where the $\Ch_\be (\cdot)$ is as in (\ref{eq:troncC}). Then, for any fixed real number $p\in (0,4]$, one has
\begin{eqnarray}
\label{eq1:main1bis}
&& c\,\ESP\Big(\big|\wh{H}_N^\be (I)-\min_{t\in I} H(t)\big|^p\Big)\nonumber  \\
&&\le \min\bigg\{\Big (\frac{\log \log N}{\log N}\Big)^p ,\la(I)^{p\rho_H}\bigg\}+N^{-p\be (\rho_H-\sup_{t\in [0,1]} H(t))}\\
&&\hspace{0.5cm} +N^{-p \Lam (L,\be,p)}(\log N)^{8(1-\Lam(L,\be,p))}\la(I)^{-4(1-\Lam (L,\be,p))}\max\Big\{\la(I)^2, N^{-(p+4)\Lam(L,\be,p)}\Big\},\nonumber
\end{eqnarray}
where 
\begin{equation}
\label{eq2:main1bis}
\Lam (L,\be,p)=\frac{2\be (L-1)}{(p+4)\be (L-1)+2(p+1)}\,,
\end{equation}
and where $c$ is a positive (deterministic) constant not depending on $N$, $I$, and $p$. 
\end{Theo}
Notice that, (\ref{eq:troncC}) implies that the random variable $\wh{H}_N^\be (I)$ is always with values in the interval $[1/2,1]$. This is coherent with the assumption that all the values of $H(\cdot)$ belongs to  
$[\underline{H},\overline{H}]\subset (1/\al,1)\subset (1/2,1)$. Also, notice that, throughout the article the natural logarithm (i.e. the Naperian logarithm) is denoted by $\log$ and the binary logarithm (i.e. to base $2$) by $\log_2$.

The following theorem is the second main result of the article. Under an additional condition (see (\ref{eq1:BorCan1}) in the theorem), it shows that not only the convergence of $\wh{H}_N^\be (I)$ to $\min_{t\in I} H(t)$ holds $L^p (\O)$, $p\in (0,4]$, but also almost surely. As well, it gives an estimate of the almost sure rate of convergence.

\begin{Theo}
\label{thm:main1}
Under the same assumptions as in Theorem~\ref{thm:main1bis} and the additional condition
%Assume that the Hurst function $H(\cdot)$ satisfies a H\"older condition as in (\ref{m:cond:holder}) and (\ref{m:cond:pholder}). Also assume that the real number $\be\in (0,1/4]$ and the integer $L\ge 2$ are arbitrary and such that
\begin{equation}
\label{eq1:BorCan1}
L> \frac{2}{\be}+1,
\end{equation}
%Let $I\subseteq [0,1]$ be an arbitrary fixed compact interval with non-empty interior and let $N$ be an integer satisfying (\ref{eq1:encadEVT2}). Then, 
there exists an almost surely finite random variable $C>0$, not depending on $N$, such that the following inequality holds almost surely
\begin{equation}
\label{eq2:main1}
\big|\wh{H}_N^\be (I)-\min_{t\in I} H(t)\big|\le C \, \frac{\log\log N}{\log N}.
\end{equation}
\end{Theo}
Let us briefly outline the four main steps of the proof of Theorem~\ref{thm:main1}; that of Theorem~\ref{thm:main1bis} follows a more or less similar strategy.
\begin{enumerate}
\item For the sake of convenience, rather than directly working with  $V_N^\be (I)$ (see  (\ref{h:def:VJ})), it is better to work with its approximation $\widetilde{V}_N^\be(I)$, defined through (\ref{h:definition:Vjtilde:lfgn1}). Lemma~\ref{h:lem:ecartdjktdjk} provides an almost sure control of the error stemming from replacing $V_N^\be(I)$ with $\wt{V}_N^\be(I)$.
\item The important Lemma~\ref{lem:maj-prob1} is the keystone of the proof. It provides, for any arbitrary fixed real number $\eta>0$, a nice upper bound for the probability 
\[
\PR\Bigg( \Big| \frac{\wt{V}_N^{\be}(I)}{\ESP \big(\wt{V}_N^{\be} (I) \big)} -1  \Big| > \eta \Bigg)\,.
\]
Notice that, in order to derive it, among other things, we make use of a Markov inequality which requires the fourth moment of $\widetilde{V}_N^\be(I)$ to be finite. This is why, the $|d_{N,k}|$'s in (\ref{h:def:VJ}) and the $|\widetilde{d}_{N,k}|$'s in (\ref{h:definition:Vjtilde:lfgn1}) are raised to the power $\be\in (0,1/4]$.
The parameter $\be$ is no further specified, for the sake of generality. 
\item  Lemma~\ref{lem:BorCan1}, whose proof relies on Lemma~\ref{lem:maj-prob1}, shows that, when $N$ goes to $+\infty$, the asymptotic behavior of the random variable $\wt{V}_N^{\be}(I)$ is, almost surely, equivalent to that of its expectation $\ESP \big(\wt{V}_N^{\be} (I)\big)$; this is, in some sense, a strong law of large numbers. Also, the lemma provides an almost sure control of the error stemming from replacing $\wt{V}_N^{\be}(I)$ with $\ESP \big(\wt{V}_N^{\be} (I)\big)$.
\item In view of the three previous steps, it turns out that, when $N$ goes to $+\infty$, the asymptotic behavior of the random ratio  $V_N^\be (I)/ V_{2N}^\be (I)$ is, almost surely, equivalent to that of the deterministic ratio $\ESP \big(\wt{V}_N^{\be} (I)\big)/\ESP \big(\wt{V}_{2N}^{\be} (I)\big)$. Also, an almost sure control of the error stemming from replacing $V_N^\be (I)/ V_{2N}^\be (I)$ with $\ESP \big(\wt{V}_N^{\be} (I)\big)/\ESP \big(\wt{V}_{2N}^{\be} (I)\big)$ is available. Notice that in the particular case of a LFSM, where the Hurst function $H(\cdot)$ is a constant denoted by $H$, using the self-similarity and the stationarity of increments of this process, it can be easily seen that $\ESP \big(\wt{V}_N^{\be} (I)\big)/\ESP \big(\wt{V}_{2N}^{\be} (I)\big)=2^{\be H}$.  In the general case of a LMSM, Lemma~\ref{lem:ratio1-esp} shows that this ratio converges to $2^{\be \underline{H} (I)}$, also the lemma provides an estimate for the rate of convergence.
\end{enumerate}

%\begin{Rem}
%\label{rem:Hmingauss}
%For the Gaussian MBM, which was introduced in \cite{roux1997elliptic}, an estimator of $\underline{H}(I)$, rather similar to $\wh{H}_N^\be (I)$, has already been given in %\cite{benassi1998identifying}.
%\end{Rem}

The second issue, we will deal with, is the statistical estimation of the whole function $H(\cdot)$. It is worth mentioning that the estimator, we will soon construct, will be almost surely and $L^p (\O)$, $p\in (0,4]$, consistent; moreover, its convergence will hold uniformly in $t\in [0,1]$, and even for some H\"older norms. Note in passing that the uniform norm on $[0,1]$ of an arbitrary real-valued function $f$, bounded on this interval, is denoted by $\|f\|_\infty$, and defined , as,
\begin{equation}
\label{eq:defUnifnorm}
\|f\|_\infty=\sup_{t\in [0,1]} |f(t)|.
\end{equation}
Also, for any fixed real number $b\in (0,1]$, we set 
\begin{equation}
\label{eq:defHolnorm}
\big\|f\big\|_b ^{\mbox{{\tiny H\"older}}}=\|f\|_{\infty}+\sup_{0\le t_1 <t_2\le 1} \frac{|f(t_1)-f(t_2)|}{|t_1-t_2|^b}\, ;
\end{equation}
when it is finite, this quantity is called the H\"older norm of $f$ of order $b$. 

Before giving a formal definition of the estimator, let us explain, in a few sentences, its way of construction. Assume that the arbitrary real number $\th_N\in (0,1/2]$ converges to zero at a convenient rate (see (\ref{eq2:main2bis}) and (\ref{eq2:main2})), when $N$ goes to $+\infty$. One sets $M_N=[\th_N^{-1}]-1$, where $[\,\cdot\,]$ is the integer part function, and one splits the  interval $[0,1]$ into a finite sequence $(\I_{N,n})_{0\le n \le M_N}$ of compact subintervals with the same length $\th_N$ (except the last one $\I_{N,M_N}$ having a length between $\th_N$ and $2\th_N$). Then, the estimator  of $H(\cdot)$ is denoted by $\wt{H}_{N,\th_N}^\be(\cdot)$, and obtained as the linear interpolation between the $M_N+2$ random points having the coordinates $$\big(0, \wh{H}_N^\be \big(\I_{N,0}\big)\big),\ldots,\big ((M_N-1)\th_N,\wh{H}_N^\be \big(\I_{N,M_N-1}\big)\big),\big (M_N\th_N,\wh{H}_N^\be \big(\I_{N,M_N}\big)\big),\big (1,\wh{H}_N^\be \big(\I_{N,M_N}\big)\big),$$
where, for all $n\in\{0,\ldots, M_n\}$, $\wh{H}_N^\be \big(\I_{N,n}\big)$ is given by (\ref{eq1:main1}) with $I=\I_{N,n}$.
Notice that the ordinate of the last point is assumed to be the same as that of the previous one. This weak assumption comes from the fact that the set of the indices $t$ of LMSM has been restricted to the interval $[0,1]$; it does not significantly alter the results on the estimation of $H(\cdot)$ on this interval.  Let us now define the estimator $\wt{H}_{N,\th_N}^\be(\cdot)$ in a formal and very precise way.
%we denote by  $\wt{H}_{N,\th_N}^\be(\cdot)$ the stochastic process whose sample paths are the functions of $\Con\big([0,1]\big)$ 
%The third and the forth main results of the article, which soon will be precisely stated, basically show that under the same assumptions as in the last two %theorems and under some conditions on $\th_N$, the unifom $\big\| H-\wt{H}_{N,\th_N}^\be\big\|_\infty$ convergence 
%The following definition provides admissible choices for $\th_N$ as well as the construction of the corresponding estimator.  

\begin{Def}
\label{def:main2}
Assume that the real number $\be\in (0,1/4]$ and the integer $L\ge 2$ are arbitrary and fixed. The integer $N_0$ is
defined as
\begin{equation}
\label{eq1:main2}
N_0=\min\big\{N\in \N : 0<9(L+1)\,N^{-1}(\log N)^2\le 1\big\},
\end{equation}
which implies that $N_0>9(L+1)$. Let $(\th_N)_{N\ge N_0}$ be an arbitrary sequence of real numbers converging to zero and satisfying, for all $N\ge N_0$,
\begin{equation}
\label{eq1bis:main2}
4(L+1)\,N^{-1}(\log N)^2\le\th_N\le 2^{-1}, 
\end{equation}
For each $n\in \big\{0,\ldots, [\th_{N}^{-1}]-1\big\}$, we denote by $\I_{N,n}$ the compact subinterval of $[0,1]$, defined as, 
\begin{equation}
\label{eq3:main2}
\mbox{$\I_{N,n}=\big [n \th_N, (n+1)\th_N\big]$ when $n< [\th_{N}^{-1}]-1$, and $\I_{N,[\th_{N}^{-1}]-1}=\big [([\th_{N}^{-1}]-1)\th_N, 1\big]$.}
\end{equation}
At last, for every integer $N\ge N_0$, we denote by $\big\{\wt{H}_{N,\th_N}^\be(t): t\in [0,1]\big\}$ the stochastic process with continuous piecewise linear paths, defined as:
\begin{equation}
\label{eq3bis:main2}
\wt{H}_{N,\th_N}^\be(t)=\wh{H}_N^\be \big(\I_{N,[\th_{N}^{-1}]-1}\big),\,\,\, 
\mbox{for all $t\in\I_{N,[\th_{N}^{-1}]-1}$,} 
\end{equation}
and, for every 
$n\in\{0,\ldots, [\th_{N}^{-1}]-2\}$ and $t\in \I_{N,n}$, as:
\begin{equation}
\label{eq4:main2}
\wt{H}_{N,\th_N}^\be(t)=\big(1-\th_N ^{-1}(t-n\th_N )\big)\wh{H}_N^\be \big(\I_{N,n}\big)+\th_N ^{-1}(t-n\th_N)\wh{H}_N^\be \big(\I_{N,n+1} \big).
\end{equation}
Notice that, for all $n\in\{0,\ldots, [\th_{N}^{-1}]-1\}$, the estimator $\wh{H}_N^\be \big(\I_{N,n}\big)$ for $\underline{H}\big(\I_{N,n}\big)$ (see (\ref{h:def1:HI})) is defined through (\ref{eq1:main1}) with $I=\I_{N,n}$. Also, notice that, in view of (\ref{eq3bis:main2}), (\ref{eq4:main2}), and the fact that $\wh{H}_N^\be \big(\I_{N,n}\big)\in [1/2,1]$, for any $n\in\{0,\ldots, [\th_{N}^{-1}]-1\}$, it can easily be shown that
\begin{equation}
\label{eq:LipfoncH}
\forall\, N\ge N_0,\,\forall\, (t_1,t_2)\in [0,1]^2, \quad \big|\wt{H}_{N,\th_N}^\be(t_1)-\wt{H}_{N,\th_N}^\be(t_2)\big|\le \th_N ^{-1}|t_1-t_2|\,.
\end{equation}
\end{Def}
The following theorem is the third main result of the article. It provides, uniformly in $p\in (0,4]$, an upper bound on the estimation error $\ESP\Big(\big\|\wt{H}_{N,\th_N}^\be-H\big\|_\infty^p\Big)$. Thus, when $N$ goes to $+\infty$, it turns out that this error converges zero, as long as the condition (\ref{eq2:main2bis}), in the theorem, is satisfied.

\begin{Theo}
\label{thm:main2bis}
We impose to $H(\cdot)$, $\be\in (0,1/4]$, and $L$ the same hypotheses as in Theorem~\ref{thm:main1bis}; moreover we use the same notations as in Definition~\ref{def:main2}. Then, there is a deterministic constant $c>0$ such that, for every real number $p\in (0,4]$, and integer $N\ge N_0$ satisfying (\ref{eq1bis:main2}), one has 
\begin{eqnarray}
\label{eq1:main2bis}
&& c\,\ESP\Big(\big\|\wt{H}_{N,\th_N}^\be-H\big\|_\infty^p\Big) \nonumber \\
&&\le (\th_N)^{p\rho_H}+N^{-p\be (\rho_H-\sup_{t\in [0,1]} H(t))}\\
&&\hspace{0.5cm} +N^{-p \Lam (L,\be,p)}(\log N)^{8(1-\Lam(L,\be,p))}\th_N^{-4(\frac{5}{4}-\Lam (L,\be,p))}\max\Big\{\th_N^2, N^{-(p+4)\Lam(L,\be,p)}\Big\},\nonumber
\end{eqnarray}
where the norm $\|\cdot\|_\infty$ is as in (\ref{eq:defUnifnorm}), and the quantity $\Lam(L,\be,p)$ as in (\ref{eq2:main1bis}). Therefore, as long as the sequence $(\th_N)_{N\ge N_0}$ satisfies 
\begin{equation}
\label{eq2:main2bis}
\lim_{N\rightarrow +\infty} \th_N =0 \mbox{ and } \lim_{N\rightarrow +\infty} N^{\frac{p\Lam (L,\be,p)}{3-4\Lam (L,\be,p)}}\,\th_N=+\infty,
\end{equation}
the estimation error $\ESP\Big(\big\|\wt{H}_{N,\th_N}^\be-H\big\|_\infty^p\Big)$ converges to zero when $N$ goes to $+\infty$. 
\end{Theo}

\begin{Rem}
\label{rem:normpointwise}
Assume that the integer $N\ge N_0$ is arbitrary. What follows, easily results from Definition~\ref{def:main2} and the triangle inequality.
\begin{itemize}
\item[(i)] For all $t\in\I_{N,[\th_{N}^{-1}]-1}$, one has,
\[
\big |\wt{H}_{N,\th_N}^\be(t)-H(t)\big |\le \big|\wh{H}_N^\be \big(\I_{N,[\th_{N}^{-1}]-1}\big)-\underline{H}(\I_{N,[\th_{N}^{-1}]-1})\big|+\big|\underline{H}(\I_{N,[\th_{N}^{-1}]-1})-H(t)\big | \,;
\]
\item[(ii)] for each $n\in \big\{0,\ldots, [\th_{N}^{-1}]-2\big\}$ and $t\in \I_{N,n}$, one has,
\begin{eqnarray*}
\big |\wt{H}_{N,\th_N}^\be(t)-H(t)\big | &\le & \big|\wh{H}_N^\be \big(\I_{N,n}\big)-\underline{H}(\I_{N,n})\big|+\big|\wh{H}_N^\be \big(\I_{N,n+1}\big)-\underline{H}(\I_{N,n+1})\big|\\
&& \hspace{1cm} +\big|\underline{H}(\I_{N,n})-H(t)\big | +\big|\underline{H}(\I_{N,n+1})-H(t)\big |\,.
\end{eqnarray*}
\end{itemize}
Then, using Definition~\ref{def:main2}, (\ref{h:def1:HI}), (\ref{m:cond:holder}), and (\ref{eq:defUnifnorm}), one gets that  
\begin{equation}
\label{eq1:normpointwise}
\big\|\wt{H}_{N,\th_N}^\be-H\big\|_\infty\le c\Big (\th_N ^{\rho_H}+ \max_{0\le n <[\th_{N}^{-1}]} \big|\wh{H}_N^\be \big(\I_{N,n}\big)-\underline{H}\big(\I_{N,n}\big)\big|\Big),
\end{equation}
where the inequality holds on the whole probability space $\O$, and $c>0$ is a deterministic constant not depending on $N$.
\end{Rem}

\begin{Rem}
\label{rem:prthmain2bis}
It can easily be seen that Theorem~\ref{thm:main2bis} results from the inequality (\ref{eq1:normpointwise}) and Proposition~\ref{lem:prmain2bis} given below.
\end{Rem}

\begin{Prop}
\label{lem:prmain2bis}
We use the same notations as in Definition~\ref{def:main2} and Theorem~\ref{thm:main2bis}. Then, under the same assumptions as in this theorem, for any integer $N\ge N_0$, one has, 
\begin{eqnarray}
\label{eq1:prmain2bis}
&& c\,\ESP\Big(\max_{0\le n <[\th_{N}^{-1}]} \big|\wh{H}_N^\be \big(\I_{N,n}\big)-\underline{H}\big(\I_{N,n}\big)\big|^p\Big) \nonumber \\
&&\le \min\bigg\{\Big (\frac{\log \log N}{\log N}\Big)^p ,(\th_N)^{p\rho_H}\bigg\}+N^{-p\be (\rho_H-\sup_{t\in [0,1]} H(t))}\\
&&\hspace{0.5cm} +N^{-p \Lam (L,\be,p)}(\log N)^{8(1-\Lam(L,\be,p))}\th_N^{-4(\frac{5}{4}-\Lam (L,\be,p))}\max\Big\{\th_N^2, N^{-(p+4)\Lam(L,\be,p)}\Big\},\nonumber
\end{eqnarray}
where $c>0$ is a deterministic constant not depending on $N$ and $p$.
\end{Prop}

The strategy of the proof of Proposition~\ref{lem:prmain2bis} will be rather similar to that of the proof of Theorem~\ref{thm:main1bis}. Let us now make a useful remark concerning the rate of convergence, for the estimator $\wt{H}_{N,\th_N}^\be(\cdot)$, provided by Theorem~\ref{thm:main2bis}.

\begin{Rem}
\label{rem:thmain2bis}
%The inequality (\ref{eq5:main2}) and the fact that the function $\be\mapsto\frac{\be(L-1)-2}{4\be(L-1)+2}$ %is increasing, show that the optimal choice of $\be$, is $\be=1/4$.
In view of the inequality (\ref{eq1:main2bis}), for the rate of vanishing of the estimation error $\ESP\Big(\big\|\wt{H}_{N,\th_N}^\be-H\big\|_\infty^p\Big)$ be a power function of $N$, one can take, for every $N\ge N_0$, 
$$
\th_N=\kap_0\, N^{-\zeta p \Lam (L,\be,p)}+4(L+1)\,N^{-1}(\log N)^2,
$$
where the normalizing constant 
\begin{equation}
\label{eq1:thmain2bis}
\kap_0=2^{-1}(L+1)\,N_0^{-1}(\log N_0)^2
\end{equation}
is chosen so that (\ref{eq1bis:main2}) holds, and where $\zeta$ is an arbitrary parameter satisfying
$
0<\zeta < \big (3-4\Lam (L,\be,p)\big)^{-1}\,. 
$
When the index $\rho_H$ (see (\ref{m:cond:holder})) is assumed to be known~\footnote{For instance, in the article \cite{benassi1998identifying}, it has been imposed to $H(\cdot)$ to be continuously differentiable, then, one has $\rho_H=1$.}, the optimal choice for $\zeta$ is then $\zeta=\big (3-4\Lam (L,\be,p)+p\rho_H \big)^{-1}$. Indeed, for the latter choice, the rate of vanishing for the estimation error, provided by (\ref{eq1:main2bis}), reduces to 
$$
\ESP\Big(\big\|\wt{H}_{N,\th_N}^\be-H\big\|_\infty^p\Big)=\dom\Big(N^{-\frac{p^2\Lam (L,\be,p)\rho_H}{3-4\Lam (L,\be,p)+p\rho_H}}(\log N)^{8(1-\Lam(L,\be,p))}+N^{-p\be(\rho_H-\sup_{t\in [0,1]} H(t))}\Big)\,.
$$
Observe that when $p=1$ and $\be L$ is large enough, then  using (\ref{eq2:main1bis}) one gets that
\[
\Lam (L,\be,1)\simeq \frac{2}{5}\quad\mbox{and}\quad\frac{\Lam (L,\be,1)\rho_H}{3-4\Lam (L,\be,1)+\rho_H}\simeq\frac{2\rho_H}{7+5\rho_H}\quad\mbox{and}\quad 8(1-\Lam(L,\be,p))\simeq \frac{24}{5}\,.
\]
\end{Rem}

The following theorem is the fourth main result of our article. Under the same assumptions as in Theorem~\ref{thm:main2bis} and the additional conditions (\ref{eq1:BorCan1}) and (\ref{eq2:main2}), it shows that, when $N$ goes $+\infty$, not only the estimation error 
$\big\|\wt{H}_{N,\th_N}^\be-H\big\|_\infty$ converges to zero in the $L^p (\O)$ spaces, with $p\in (0,4]$, but also almost surely. As well, it gives an estimate of the almost sure rate of convergence.

 \begin{Theo}
\label{thm:main2}
We suppose that $H(\cdot)$, $\be$, and $L$ satisfy the same hypotheses as in Theorem~\ref{thm:main1bis}, and that (\ref{eq1:BorCan1}) holds; moreover we use the same notations as in Definition~\ref{def:main2}. Let us assume in addition that
\begin{equation}
\label{eq2:main2}
\lim_{N\rightarrow +\infty} \th_N =0 \mbox{ and } \lim_{N\rightarrow +\infty} N^{\frac{\be(L-1)-2}{4\be(L-1)+2}}\,\th_N=+\infty.
\end{equation}
Then, there exists a positive and almost surely finite random variable $C>0$, such that one has almost surely for all integer $N\ge N_0$ 
\begin{equation}
\label{eq5:main2}
\big\|\wt{H}_{N,\th_N}^\be-H\big\|_\infty \le C\bigg (\th_N ^{-1} N^{-\frac{\be(L-1)-2}{4\be(L-1)+2}}+N^{-\be (\rho_H-\sup_{t\in [0,1]} H(t))}+\th_N ^{\rho_H}\bigg). 
\end{equation}
\end{Theo}

\begin{Rem}
\label{rem:barsur}
In a Gaussian setting more general than that of MBM, almost surely uniformly convergent estimators of the whole Hurst function $H(\cdot)$ have already been obtained in \cite{bardet2010nonparametric}.
\end{Rem}
\begin{Rem}
\label{rem:prthmain2}
It can easily be seen that Theorem~\ref{thm:main2} results from the inequality (\ref{eq1:normpointwise}) and Proposition~\ref{lem:maxhaHhtH} given below.
\end{Rem}

\begin{Prop}
\label{lem:maxhaHhtH}
We use the same notations as in Definition~\ref{def:main2} and Theorem~\ref{thm:main2}. Then, under the same assumptions as in this theorem, there exists a finite random variable $C>0$, such that one has almost surely for all $N$ large enough,
\begin{eqnarray}
\label{eq1:maxhaHhtH}
&& \max_{0\le n <[\th_{N}^{-1}]} \big|\wh{H}_N^\be \big(\I_{N,n}\big)-\underline{H}\big(\I_{N,n}\big)\big|\\
&& \le C \bigg (\th_N ^{-1} N^{-\frac{\be(L-1)-2}{4\be(L-1)+2}}+N^{-\be(\rho_H-\sup_{t\in [0,1]} H(t))}+\min\Big\{\frac{\log \log N}{\log N},\th_N^{\rho_H}\Big\}\bigg)\,. \nonumber
\end{eqnarray}
\end{Prop}
%\noindent {\bf Remarks}
%\begin{enumerate}
%\item The parameter $\be\in (0,1/4]$ has been introduced for the sake of generality. When $\al=2$ i.e. in %the Gaussian case Theorems~\ref{thm:main1} and  \ref{thm:main2} remain valid even if $\be$ belongs 
%\end{enumerate}

%\begin{Rem}
%\label{Remh:thm:main1}
%\begin{enumerate}
%\item  A careful inspection of the proof of Theorem~\ref{thm:main1}, shows that it remains valid in the %Gaussian setting of Multifractional Brownian Motion (MBM); in this case, $\al=2$, one takes $\beta=2$, and %$H(\cdot)$ is allowed to be with values in any compact subinterval of $(0,1)$. Even in this quite classical %framework, the theorem provides some interesting results which were unknown so far, namely: {\bf %(d\'evelopper)}
%\item 
%\end{enumerate}
%\end{Rem}
The strategy of the proof of Proposition~\ref{lem:maxhaHhtH} will be rather similar to that of the proof of Theorem~\ref{thm:main1}.
Let us now make a useful remark concerning the rate of convergence, for the estimator $\wt{H}_{N,\th_N}^\be(\cdot)$, provided by Theorem~\ref{thm:main2}.

\begin{Rem}
\label{rem:thmain2}
%The inequality (\ref{eq5:main2}) and the fact that the function $\be\mapsto\frac{\be(L-1)-2}{4\be(L-1)+2}$ %is increasing, show that the optimal choice of $\be$, is $\be=1/4$.
In view of (\ref{eq5:main2}), for the almost sure rate of vanishing of the estimation error $\big\|\wt{H}_{N,\th_N}^\be-H\big\|_\infty$ be a power function of $N$, one can take, for every $N\ge N_0$,
$$
\th_N=\kap_0\, N^{-\frac{\zeta(\be(L-1)-2)}{4\be(L-1)+2}}+4(L+1)\,N^{-1}(\log N)^2,
$$
where the normalizing constant $\kap_0$ is as in (\ref{eq1:thmain2bis}), and where $\zeta$ is an arbitrary parameter belonging to the interval $(0,1)$. When the index $\rho_H$ (see (\ref{m:cond:holder})) is assumed to be known, the optimal choice for $\zeta$ is then $\zeta=(1+\rho_H)^{-1}$. Indeed, for the latter choice, the almost sure rate of vanishing for estimation error, provided by (\ref{eq5:main2}), reduces to  
$$
\big\| H-\wt{H}_{N,\th_N}^\be\big\|_\infty=\dom\Big(N^{-\frac{\rho_H (\be(L-1)-2)}{(1+\rho_H)(4\be(L-1)+2)}}+N^{-\be(\rho_H-\sup_{t\in [0,1]} H(t))}\Big),
$$
Observe that when $\be L$ is large enough, then 
\[
\frac{\rho_H (\be(L-1)-2)}{(1+\rho_H)(4\be(L-1)+2)}\simeq \frac{\rho_H }{4(1+\rho_H)}\,.
\]
The latter estimate, compared with the ones given at the end of Remarks~\ref{rem:thmain2bis}, somehow reveals that the almost sure rate of vanishing of the estimation error $\big\| H-\wt{H}_{N,\th_N}^\be\big\|_\infty$ is slower than its $L^1(\O)$ rate of convergence. This is not very surprising, since, in general, convergence in $L^p (\O)$, $p\in (0,+\infty)$, is less demanding than almost sure convergence.
\end{Rem}

%When we only require the estimator $\wt{H}_{N,\th_N}^\be$ to converge, for some norm, to the function $H$, at a logarithmic rate, then the choice of this norm can %be more ambitious than $\|\cdot\|_\infty$ the uniform norm: some H\"older norms $\|\cdot\|_{\ce^b}$ can work. We recall that, for any real number $b\in (0,1)$, %a real-valued function $f$, continuous on the interval $[0,1]$, belongs to the H\"older space $\ce^b\big ([0,1]\big)$, when $\|f\|_{\ce^b}$, defined as 

So far, we have shown that the almost sure and $L^p (\O)$, $p\in (0,4]$, consistency of the estimator $\wt{H}_{N,\th_N}^\be(\cdot)$ holds in terms of the uniform norm $\|\cdot\|_\infty$, defined through (\ref{eq:defUnifnorm}). In fact, without a lot of extra work, this consistency can also be obtained, for some indices $b$, in terms of the H\"older norm $\big\|\cdot\big\|_{b}^{\mbox{{\tiny H\"older}}}$, defined through  (\ref{eq:defHolnorm}). Moreover the bounds (\ref{eq1:main2bis}) and (\ref{eq5:main2}), for controlling the $L^p(\O)$ and almost sure rates of convergence of $\wt{H}_{N,\th_N}^\be(\cdot)$, can be extended to the setting of this norm. In order to do so, one still has a lot of freedom one the choice of the form of the sequence $(\th_N)_{N\ge N_0}$; however, for the sake of simplicity, we will assume, in the sequel, that it is of a logarithmic form.

The following result is a consequence of (\ref{eq:LipfoncH}) and of Theorem~\ref{thm:main2bis}.
\begin{Cor}
\label{cor:cmain2bis}
We use the same notations as in Theorem~\ref{thm:main2bis} and Definition~\ref{def:main2}. We assume that the hypotheses of this theorem are satisfied, and that, for all $N\ge N_0$, 
\begin{equation}
\label{eq1:cmain2bis}
\th_N=\kap_0\big (\log N)^{-1}+4(L+1)\,N^{-1}(\log N)^2,
\end{equation}
where the constant $\kap_0$ was defined in (\ref{eq1:thmain2bis}). Let $\rho_H$ be as in (\ref{m:cond:holder}) and $b$ be an arbitrary fixed real number such that 
\begin{equation}
\label{eq2:cmain2bis}
0< b <\frac{\rho_H}{1+\rho_H}.
\end{equation}
Then, for any fixed real numbers $\be\in (0,1/4]$ and $p\in (0,4]$, the estimation error $\ESP\bigg(\Big(\big\|\wt{H}_{N,\th_N}^\be-H\big\|_{b}^{\mbox{{\tiny H\"older}}}\Big)^p\bigg)$ vanishes at a logarithmic rate, when $N$ goes to $+\infty$.
\end{Cor}

The following result is a consequence of (\ref{eq:LipfoncH}) and of Theorem~\ref{thm:main2}.

\begin{Cor}
\label{cor:cmain2}
We use the same notations as in Theorem~\ref{thm:main2} and Definition~\ref{def:main2}. We assume that the hypotheses of this theorem are satisfied, and that $\th_N$ is defined through (\ref{eq1:cmain2bis}), for all $N\ge N_0$. Let $b$ be an arbitrary fixed real number satisfying (\ref{eq2:cmain2bis}).  
Then, for any fixed real numbers $\be\in (0,1/4]$, the estimation error $\big\|\wt{H}_{N,\th_N}^\be-H\big\|_{b}^{\mbox{{\tiny H\"older}}}$ vanishes almost surely at a logarithmic rate, when $N$ goes to $+\infty$.
\end{Cor}

We will only give the proof of Corollary~\ref{cor:cmain2} since that of Corollary~\ref{cor:cmain2bis} can be done in a similar way.

\begin{proof}[Proof of Corollary~\ref{cor:cmain2}] Let $w_0$ be the real number, strictly greater than $1$, defined as $w_0=1+\rho_H$. We know from the second inequality in (\ref{eq2:cmain2bis}) that 
$\rho_H >w_0 b$ and $w_0 (1-b)>1$. Thus, in view of (\ref{eq5:main2}) and (\ref{eq1:cmain2bis}), in order to obtain the corollary, it is enough to prove that,
almost surely, one has,
\begin{equation}
\label{eq1:cmain2}
\big\|\wt{H}_{N,\th_N}^\be-H\big\|_{b}^{\mbox{{\tiny H\"older}}}=\dom\Big(\th_N ^{w_0 (1-b)-1}+\th_N ^{w_0(\rho_H-b)}+\th_N ^{-w_0 b}\,\big\|\wt{H}_{N,\th_N}^\be-H\big\|_{\infty}\Big).
\end{equation}
To this end, in view of (\ref{eq:defHolnorm}), one needs to conveniently bound, from above, the quantity 
\begin{equation}
\label{eq2:cmain2}
\varphi_b (t_1,t_2)=\frac{\big | \wt{H}_{N,\th_N}^{\be} (t_1)-\wt{H}_{N,\th_N}^\be (t_2)-H(t_1)+H(t_2)\big | }{|t_1-t_2|^b}\,,
\end{equation}
where $t_1$ and $t_2$ are two arbitrary real numbers satisfying $0\le t_1 < t_2\le 1$. It follows from (\ref{eq2:cmain2}), (\ref{eq:LipfoncH}), and (\ref{m:cond:holder}), that
\[
\varphi_b (t_1,t_2)\le \th_N ^{-1}\,|t_1-t_2|^{1-b}+ c\,|t_1-t_2|^{\rho_H -b},
\]
where $c$ is the same constant as in (\ref{m:cond:holder}). Therefore, we get that 
\begin{equation}
\label{eq3:cmain2}
\varphi_b (t_1,t_2)\le \th_N ^{w_0 (1-b)-1} +c\,\th_N ^{w_0(\rho_H -b)}, \quad\mbox{when $|t_1-t_2|<\th_N ^{w_0}$.} 
\end{equation}
On the other hand, (\ref{eq2:cmain2}), (\ref{eq:defUnifnorm}), and the triangle inequality clearly imply that
\[
\varphi_b (t_1,t_2)\le 2 |t_1-t_2|^{-b}\,\big\|\wt{H}_{N,\th_N}^\be-H\big\|_{\infty}\,.
\]
Therefore, we obtain that
\begin{equation}
\label{eq4:cmain2}
\varphi_b (t_1,t_2)\le 2\th_N ^{-w_0 b}\,\big\|\wt{H}_{N,\th_N}^\be-H\big\|_{\infty}\,, \quad\mbox{when $|t_1-t_2|\ge\th_N ^{w_0}$.} 
\end{equation}
Finally, putting together (\ref{eq:defHolnorm}), (\ref{eq2:cmain2}), (\ref{eq3:cmain2}), and (\ref{eq4:cmain2}), it follows that (\ref{eq1:cmain2}) holds.
\end{proof}

The remaining of our article is organized in the following way. The third section provides the keystone of the proofs of the main results, and the other two sections are devoted to their proofs in themselves.
%The rate of convergence in (\ref{eq2:thmain2}) can be improved under a regularity assumption on $H(\cdot)$ stronger than (\ref{m:cond:holder}); namely, %one has:

 %\begin{Theo}
%\label{thm:main3}
%Assume that the real number $\be\in (0,1/4]$ and the integer $L$ are arbitrary and satisfy the inequality (\ref{eq1:BorCan1}); also assume that
%\begin{equation}\label{eq1:main3}
%\left\{
%\begin{array}{l}
%\mbox{$H(\cdot)$ is $[1+(8\be)^{-1}]$-times continuously differentiable}\\
%\mbox{and its $[1+(8\be)^{-1}]$-th derivative satisfies}\\
%\mbox{a H\"older condition of order $1+(8\be)^{-1}-[1+(8\be)^{-1}]$.}
%\end{array}
%\right.
%\end{equation}
%Then,
%there exists a positive and finite random variable $C>0$, such that one has almost surely for all integer $N$ big enough 
%\begin{equation}
%\label{eq2:main3}
%\big\| H-\wt{H}_{N,\th_N}^\be\big\|_\infty=\le C\, N^{-\frac{\be(L-1)-2}{8\be(L-1)+4}}. 
%\end{equation} 
%Thus, for every fixed arbitrarily small $\ep>0$, taking $L$ large enough, one has, almost surely, 
%$$
%\big\| H-\wt{H}_{N,\th_N}^\be\big\|_\infty=\dom\big(N^{-8^{-1}+\ep}\big).
%$$
%\end{Theo}

%Before ending the present section, let us mention that in the sequel, 

\section{The Keystone}
\label{sec:keystone}

%\begin{equation}
%\label{prob0}
%\PR\Bigg( \Big| \frac{V_N^\be(I)}{\ESP \big(V_N^\be (I) \big)} -1  \Big| > \eta \Bigg).
%\end{equation}
%A natural strategy for proving Theorem~\ref{h:thm:main1}, variants of which have already been used several %times in the literature, consists in the following two main steps:
%\begin{enumerate}
%\item to establish that, when $j$ goes to $+\infty$, the asymptotic behavior of the empirical mean $V_j$, is %equivalent to that of its expectation $\ESP(V_j)$, namely one has,
%\begin{equation}
%\label{eqa:ratioVj}
 %\frac{V_j}{\ESP \big( V_j \big)} \xrightarrow[j\rightarrow+\infty]{a.s.} 1;
 %\end{equation}
 %\item to show that (\ref{h:main:equation}) holds when $V_j$ is replaced by $\ESP(V_j)$.
 %\end{enumerate}
Let $I\subseteq [0,1]$ be an arbitrary compact interval having a positive Lebesgue measure $\la (I)$ (in other words, $I$ has a non-empty interior), and let $N$ be an arbitrary integer such that $N\ge (L+1)\la(I)^{-1}$. The non-empty set $\nu_N (I)$ is as in (\ref{h:def:nuI}), and $|\nu_N (I)|$ is its cardinality. We denote by $\widetilde{V}_N^\be(I)$ the approximation of $V_N^\be(I)$ (see (\ref{h:def:VJ})) defined as
\begin{equation}\label{h:definition:Vjtilde:lfgn1}
\widetilde{V}_N^\be(I)=|\nu_N (I)|^{-1} \sum_{k\in \nu_N (I)} |\wt{d}_{N,k}|^{\beta},
\end{equation}
where
\begin{equation}\label{h:tdjk}
\widetilde{d}_{N,k}=\sum_{l=0}^L a_l X\big((k+l)/N,H(k/N)\big).
\end{equation}
The main goal of the present section is to show the following lemma; let us point out that this lemma is the keystone of the proofs of the main results of the article. 
\begin{Lem}
\label{lem:maj-prob1}
For any fixed integer $L\ge 2$ and real number $\be\in (0,1/4]$, there is a constant $c>0$ satisfying the following property: for all compact interval $I\subseteq [0,1]$ with non-empty interior, and for each real numbers $\de\in (0,1)$ and $\eta>0$, the inequality 
\begin{eqnarray}
\label{eq1:maj-prob1}
\PR\Bigg( \Big| \frac{\wt{V}_N^{\be}(I)}{\ESP \big(\wt{V}_N^{\be} (I) \big)} -1  \Big| > \eta \Bigg)&\le & c\, \eta^{-1} N^{-\de\be (L-\underline{H}(I))}(\log N)^{2}\\\
&& \hspace{0.4cm}+c\, \eta^{-4}\bigg (\frac{N^{-2(1-\de)}}{\la(I)^{2}}+\frac{N^{-4(1-\de)}}{\la(I)^{4}}+N^{-4\de\be (L-\underline{H}(I))}\bigg)(\log N)^{8}\nonumber
\end{eqnarray}
holds, for every integer $N\ge 4$ which satisfies (\ref{eq1:encadEVT2}). Recall that $\underline{H}(I)$ is defined through (\ref{h:def1:HI}).
\end{Lem}

Before focusing on the proof of the lemma, we precisely explain the reason why it is better to work with $\widetilde{V}_N^\be(I)$ rather than with $V_N^\be(I)$, also we provide an almost sure control of the error stemming from replacing $V_N^\be(I)$ with $\widetilde{V}_N^\be(I)$. The $\wt{d}_{N,k}$'s, through which $\widetilde{V}_N^\be(I)$ is defined, offer the advantage to have a stable stochastic integral representation which is rather convenient to handle. More precisely, one can derive from (\ref{h:def:champX}), (\ref{h:tdjk}), (\ref{eq1:moma}), (\ref{h:eq:pospart}) and easy computations, that 
\begin{equation}\label{h:tdjk2}
\widetilde{d}_{N,k}= N^{-(H(k/N)-1/\al)} \int_{-\infty}^{N^{-1}(k+L)} \Phi_{\al}\big(N s-k,H(k/N)\big) \Za{ds},
\end{equation}
where $\Phi_{\al}$ is the real-valued continuous function, defined, for all $(u,v)\in \R \times (1/\al,1)$, as,
\begin{equation}\label{h:def:phial}
\Phi_{\al}(u,v) = \sum_{l=0}^L a_l (l-u)_+^{v-1/\al}.
\end{equation}
Observe that for any fixed  $v\in(1/\al,1)$, the function $\Phi_{\al}(\cdot,v)   $ can almost be viewed as a compactly supported function in the variable $u$, since it satisfies the following two nice localization properties.
\begin{Prop}
\label{propj:locaphial}
One has
\begin{equation}
\label{eq1:locaphial}
\supp\big(\Phi_{\al}(\cdot,v)\big) \subseteq ]-\infty, L],\,\mbox{for all fixed $v\in(1/\al,1)$,} 
\end{equation}
and
\begin{equation}
\label{eq2:locaphial}
c=\sup\left\{\big(1+L+|u|\big)^{L+1/\al-v}\,\big|\Phi_{\al}(u,v) \big| \,:\, (u,v)\in ]-\infty, L]\times (1/\al,1)\right\}<+\infty.
\end{equation}
\end{Prop}
\begin{proof}
The inclusion (\ref{eq1:locaphial}) is a straightforward consequence of (\ref{h:def:phial}) and (\ref{h:eq:pospart}). The inequality (\ref{eq2:locaphial}) holds when the supremum is restricted to $(u,v)\in [-2L, L]\times (1/\al,1)$, since by using (\ref{h:def:phial}) and (\ref{h:eq:pospart}), one gets that
$$
\big(1+L+|u|\big)^{L+1/\al-v}\,\big|\Phi_{\al}(u,v) \big|\le (1+3L)^{L}\sum_{l=0}^L |a_l|<+\infty.
$$
So, it remains to show that the inequality (\ref{eq2:locaphial}) holds when the supremum is restricted to $(u,v)\in (-\infty, -2L)\times (1/\al,1)$. By using (\ref{h:def:phial}) and (\ref{h:eq:pospart}) one has,
\begin{equation}
\label{eq3:locaphial}
\big|\Phi_{\al}(u,v) \big|=|u|^{v-1/\al} \sum_{l=0}^L a_l f(l u^{-1},v),
\end{equation}
where $f$ is the $C^\infty$ function on $(-1,1)\times (-2,2)$, defined for all $(y,v)\in (-1,1)\times (-2,2)$, as $f(y,v)=(1-y)^{v-1/\al}$. Then noticing that $z=u^{-1}$ belongs to  $\big (-2^{-1}L^{-1} ,0\big)$, one can easily derive from (\ref{eq3:locaphial}) and Lemma~\ref{lemj:taylor} below, that
$$
\sup\left\{\big(1+L+|u|\big)^{L+1/\al-v}\,\big|\Phi_{\al}(u,v) \big| \,:\, (u,v)\in ]-\infty, -2L)\times (1/\al,1)\right\}<+\infty.
$$
\end{proof}
\begin{Lem}
\label{lemj:taylor}
Assume that $y_0$ and $v_0$ are two arbitrary and fixed positive real numbers. Let $f$ be an arbitrary real-valued $C^\infty$ function on $(-y_0 , y_0)\times(-v_0, v_0)$ and let $g$ be the $C^\infty$ function on $\big(-L^{-1} y_0,L^{-1} y_0\big)\times(-v_0, v_0)$ defined for all $(z,v)\in \big(-L^{-1} y_0,L^{-1} y_0\big)\times(-v_0, v_0)$ as 
$$
g(z,v)=\sum_{l=0}^L a_l f(l z,v).
$$ 
Then, one has for every $(z,v)\in\big [-2^{-1}L^{-1} y_0,2^{-1}L^{-1} y_0\big]\times \big[-2^{-1}v_0, 2^{-1}v_0\big]$,
$$
\big |g(z,v)\big|\le c |z|^L,
$$
where $c$ is the finite constant defined as 
$$
c=(L!)^{-1}\sup\big\{|(\partial_z^L g)(z,v)|:(z,v)\in\big [-2^{-1}L^{-1} y_0,2^{-1}L^{-1} y_0\big]\times \big [-2^{-1}v_0, 2^{-1}v_0\big]\big\}. 
$$
\end{Lem}
\begin{proof} Assume that $v\in (-v_0,v_0)$ is arbitrary and fixed. Applying Taylor formula to the function $z\mapsto g(z,v)$ it follows that for all 
$z\in \big [-2^{-1}L^{-1} y_0,2^{-1}L^{-1} y_0\big]$
\begin{equation}
\label{eqj1:taylor}
g(z,v)=\Big(\sum_{q=0}^{L-1}\frac{(\partial_z ^q g)(0,v)}{q!} z^l \Big)+\frac{(\partial_z ^L g)(\th,v)}{L!} z^L,
\end{equation} 
where $\th\in \big (-2^{-1}L^{-1} y_0,2^{-1}L^{-1} y_0\big)$. Next, observe that for each $z\in \big(-L^{-1} y_0,L^{-1} y_0\big)$ and $q\in\N$ one has
$$
(\partial_z ^q g)(z,v)=\sum_{l=0}^L l^q a_l (\partial_y^q f)(l z,v),
$$
therefore
$$
(\partial_z ^q g)(0,v)= (\partial_y^q f)(0,v)\big(\sum_{l=0}^L l^q a_l\big).
$$
Thus, in view of (\ref{eq1:moma}), for all $q\in\{0,\ldots , L-1\}$, $(\partial_z ^q g)(0,v)=0$. Finally, combining the latter equality with (\ref{eqj1:taylor}), one gets the lemma.
\end{proof}
The following lemma provides, independently on $I$, an almost sure control of the error stemming from replacing $V_N^\be(I)$ with $\wt{V}_N^\be(I)$.

\begin{Lem}\label{h:lem:ecartdjktdjk}
There exists a positive random variable $C_*$, not depending on $I$, such that the inequality 
\begin{equation}\label{h:ecartdjktdjk}
\big|V_N^\be(I)-\wt{V}_N^\be(I)\big| \le C_*^\be N^{-\be\rho_H},
\end{equation}
holds, for all $N\ge 2(L+1)\la(I)^{-1}$, on the whole probability space $\O$. Moreover, one has, for some finite constant $c_*>0$ and every positive real number $z$, the following control on the tail of 
the distribution of $C_*$,
\begin{equation}
\label{eq2:taildistA}
\PR(C_*>z)\le c_*\,z^{-\al}.
\end{equation}
As a straightforward consequence one has $\ESP (C_*^q)<+\infty$, for each real number $q\in [0,\al)$.
\end{Lem}

\begin{proof} First observe that using (\ref{h:def:VJ}), (\ref{h:definition:Vjtilde:lfgn1}) and the inequality
\begin{equation}
\label{trian-ineg-be}
\forall\, x,y\in\R,\,\,\, \big||x|^\be-|y|^\be\big |\le |x-y|^\be,
\end{equation}
one gets that 
\begin{equation}
\label{eq1:ecartdjktdjk}
\big|V_N^\be(I)-\wt{V}_N^\be(I)\big| \le |\nu_N (I)|^{-1} \sum_{k\in \nu_N (I)} |d_{N,k}-\wt{d}_{N,k}|^{\beta}.
\end{equation}
On the other hand, in view of (\ref{m:djk}) and (\ref{h:tdjk}), one has 
\begin{equation*}
\label{eq2:ecartdjktdjk}
d_{N,k}-\wt{d}_{N,k}=\sum_{l=0}^L a_l\Big (X\big((k+l)/N,H((k+l)/N)\big)-X\big((k+l)/N,H(k/N)\big)\Big).
\end{equation*}
Thus, (\ref{eqa:uniflipX}) implies that 
\begin{equation}
\label{eq3:ecartdjktdjk}
|d_{N,k}-\wt{d}_{N,k}|\le A\sum_{l=0}^L |a_l| \big|H((k+l)/N)-H(k/N)\big|.
\end{equation}
Next, putting together (\ref{m:cond:holder}) and (\ref{eq1:ecartdjktdjk}) to (\ref{eq3:ecartdjktdjk}), it follows that the inequality (\ref{h:ecartdjktdjk}) is satisfied when 
 \begin{equation}
\label{eq4:ecartdjktdjk}
C_*=c A \Big (\sum_{l=1}^L l |a_l|\Big),
\end{equation}
where $c$ is the same constant as in (\ref{m:cond:holder}). Finally, in view of (\ref{eq:taildistA}) and (\ref{eq4:ecartdjktdjk}), it is clear that (\ref{eq2:taildistA}) holds.
%Rougly speaking, the inequality (\ref{hbis:ecartdjktdjk}) can be obtained in the way. However, in order to obtain an appropriate estimate for %$d_{N,k}-\wt{d}_{N,k}$, instead of using the mean-value Theorem, one has, for each fixed $k$, $l$ and $\o\in\O$, to apply $[1+(8\be)^{-1}]$-th order %Taylor formula, to the function 
%$$
%x\mapsto  X\big((k+l)/N,H(x),\o\big),
%$$
%on the interval whose endpoints are $(k+l)/N$ and $k/N$; then, in the expression of $d_{N,k}-\wt{d}_{N,k}$, obtained in this way, the polynomial part in %$l/N$, vanishes because of (\ref{eq1:moma}).
\end{proof}
From now on, we focus on the proof of Lemma~\ref{lem:maj-prob1}; first, we explain the main idea behind it.
{\em The main difficulty in the proof is that the random variables $\wt{d}_{N,k}$ (see (\ref{h:tdjk2})) have a complex dependence structure.} In order to precisely explain our main idea for overcoming this difficulty, we need to introduce a few notations. Assume that $\de\in (0,1)$ is fixed and that $e_N=e_N(\de)$ is the positive integer defined as
\begin{equation}
\label{m:eq:add1}
e_N=\big[N^{\delta}\big].
\end{equation}
In view of  (\ref{h:tdjk2}), $\wt{d}_{N,k}$ can be expressed as
\begin{equation}
\label{eq:expr-dtjN}
\wt{d}_{N,k}=\wt{d}_{N,k}^{\,1,\de}+\wt{d}_{N,k}^{\,2,\de}\,,
\end{equation}
where 
\begin{equation}
\label{eq:def-dtjN1}
\wt{d}_{N,k}^{\,1,\de}= N^{-(H(k/N)-1/\al)} \int_{N^{-1}(k-e_N+L)}^{N^{-1}(k+L)} \Phi_{\al}\big(N s-k,H(k/N)\big) \Za{ds}
\end{equation}
and 
\begin{equation}
\label{eq:def-dtjN2}
\wt{d}_{N,k}^{\,2,\de}= N^{-(H(k/N)-1/\al)} \int_{-\infty}^{N^{-1}(k-e_N+L)} \Phi_{\al}\big(N s-k,H(k/N)\big) \Za{ds}.
\end{equation}
In order to clarify the importance of the decomposition (\ref{eq:expr-dtjN}), let us introduce  for each fixed $r\in\{0,\ldots, e_N-1\}$, the set 
of indices $\jnr$ defined as 
\begin{equation}
\label{eq:INr}
\jnr=\big\{k\in \nu_N (I): \exists\, q\in\Z_+ \mbox{ s.t. } k=qe_N+r\big\};
\end{equation}
in other words, $\jnr$ is the set of the $k$'s belonging to $\nu_N (I)$ which are congruent to $r$ modulo $e_N$. Observe that 
\begin{equation}
\label{eq3:newlem1}
\nu_N (I)=\bigcup_{r=0}^{e_N-1} \jnr \quad\mbox{(disjoint union),}
\end{equation}
also observe that the fact that $\de\in (0,1)$ implies that all the $\jnr$'s are non-empty when $N$ is large enough.

{\em The decomposition (\ref{eq:expr-dtjN}) will play a key role in the proof of Lemma~\ref{lem:maj-prob1}, since, for each fixed $r\in\{0,\ldots, e_N-1\}$,
$\{\wt{d}_{N,k}^{\,1,\de}: k\in \jnr\}$ is a finite sequence of independent random variables. Thus, setting
\begin{equation}\label{req:Vjtilde1}
\widetilde{V}_N^{\be,1,\de}(I)=|\nu_N (I)|^{-1} \sum_{k\in \nu_N (I)} |\wt{d}_{N,k}^{\,1,\de}|^{\beta},
\end{equation}
it is less difficult to bound from above the probability 
\begin{equation}
\label{prob2}
\PR\Bigg( \Big| \frac{\wt{V}_N^{\be,1,\de}(I)}{\ESP \big(\wt{V}_N^{\be} (I) \big)} -1  \Big| > \eta \Bigg),
\end{equation}
than to obtain, in a direct way, an upper bound for the probability 
\[
\PR\Bigg( \Big| \frac{\wt{V}_N^{\be}(I)}{\ESP \big(\wt{V}_N^{\be} (I) \big)} -1  \Big| > \eta \Bigg),
\]
which figures in the left-hand side of (\ref{eq1:maj-prob1}).} An appropriate upper bound of the probability (\ref{prob2}) is provided by Lemma~\ref{lem:maj-prob2} below. Moreover, for all $\eta >0$, one has
\begin{equation}
\label{eq:ineq3prob}
\PR\Bigg( \Big| \frac{\wt{V}_N^{\be}(I)}{\ESP \big(\wt{V}_N^{\be} (I) \big)} -1  \Big| > \eta \Bigg)\le\PR\Bigg( \Big| \frac{\wt{V}_N^{\be,1,\de}(I)}{\ESP \big(\wt{V}_N^{\be} (I) \big)} -1  \Big| > \frac{\eta}{2} \Bigg)+\PR\Bigg( \frac{\wt{V}_N^{\be,2,\de}(I)}{\ESP \big(\wt{V}_N^{\be} (I) \big)}  > \frac{\eta}{2} \Bigg).
\end{equation}
where
\begin{equation}\label{req:Vjtilde2}
\widetilde{V}_N^{\be,2,\de}(I)=|\nu_N (I)|^{-1} \sum_{k\in \nu_N (I)} |\wt{d}_{N,k}^{\,2,\de}|^{\beta}.
\end{equation}
Notice that (\ref{eq:ineq3prob}) easily follows from the inequalities 
$$
\wt{V}_N^{\be,1,\de}(I)-\wt{V}_N^{\be,2,\de}(I)\le \wt{V}_N^{\be}(I)\le \wt{V}_N^{\be,1,\de}(I)+\wt{V}_N^{\be,2,\de}(I),
$$
which in turn are obtained by using
the triangle inequality, (\ref{h:definition:Vjtilde:lfgn1}), (\ref{eq:expr-dtjN}), (\ref{req:Vjtilde1}), (\ref{req:Vjtilde2}),  (\ref{trian-ineg-be}), and the inequality
\begin{equation}
\label{trian-ineg-be2}
\forall\, x,y\in\R,\,\,\, \big(|x|+|y|\big)^\be\le |x|^\be+|y|^\be.
\end{equation}
The following remark provides the way to obtain Lemma~\ref{lem:maj-prob1}.
\begin{Rem}
\label{rem:proofkeystone}
Lemma~\ref{lem:maj-prob1} is a straightforward consequence 
of the inequality (\ref{eq:ineq3prob}) and the following two lemmas.
\end{Rem}
\begin{Lem}
\label{lem:maj-prob2}
For any fixed integer $L\ge 2$ and real number $\be\in (0,1/4]$, there is a constant $c>0$ having the following property: for all compact interval $I\subseteq [0,1]$ with non-empty interior, and for each real numbers $\de\in (0,1)$ and $\eta>0$, the inequality 
\begin{equation}
\label{eq1:maj-prob2}
\PR\Bigg( \Big| \frac{\wt{V}_N^{\be,1,\de}(I)}{\ESP \big(\wt{V}_N^{\be} (I) \big)} -1  \Big| > \eta \Bigg)\le c\, \eta^{-4}\,\bigg (\frac{N^{-2(1-\de)}}{\la(I)^{2}}+\frac{N^{-4(1-\de)}}{\la(I)^{4}}+N^{-4\de\be (L-\underline{H}(I))}\bigg)(\log N)^{8}
\end{equation}
holds, for every integer $N\ge 4$ which satisfies (\ref{eq1:encadEVT2}).
\end{Lem}

\begin{Lem}
\label{lem:maj-prob3}
For any fixed integer $L\ge 2$ and real number $\be\in (0,1/4]$, there is a constant $c>0$ satisfying the following property: for all compact interval $I\subseteq [0,1]$ with non-empty interior, and for each real numbers $\de\in (0,1/2]$ and $\eta>0$, the inequality 
\begin{equation}
\label{eq1:maj-prob3}
\PR\Bigg( \frac{\wt{V}_N^{\be,2,\de}(I)}{\ESP \big(\wt{V}_N^{\be} (I) \big)}  > \eta \Bigg)\le c\,\eta^{-1}N^{-\de\be (L-\underline{H}(I))}(\log N)^{2}
\end{equation}
holds, for every integer $N\ge 4$ which satisfies (\ref{eq1:encadEVT2}).
\end{Lem}

From now on, our goal is to show that Lemmas~\ref{lem:maj-prob2} and \ref{lem:maj-prob3} hold. To this end, we need some preliminary results. We mention in passing that 
for obtaining  Lemma~\ref{lem:maj-prob2} more work is required than for Lemma~\ref{lem:maj-prob3}.

%The intuition behind the following lemma is that the scale parameter $\|\widetilde{d}_{N,k} \|_{\al}$ is almost equal to $N^{- H(k/N)}$.

\begin{Lem}\label{h:lem:borne:scaletdjk}
There exist two constants $0<c_{2}\le c_1$, such that for every integers $N\ge L$ and $0\le k \le N-L$, one has,
\begin{equation}\label{h:borne:scaletdjk}
c_{2} N^{- H(k/N)} \le \|\widetilde{d}_{N,k} \|_{\al} \le c_{1} N^{- H(k/N)},
\end{equation}
where $\|\widetilde{d}_{N,k} \|_{\al}$ is the scale parameter of the $\sas$ random variable $\widetilde{d}_{N,k}$.
\end{Lem}

The proof of Lemma~\ref{h:lem:borne:scaletdjk} mainly relies on the following lemma.
\begin{Lem}
\label{rem:minmaxphial}
The nonnegative function $v \mapsto \| \Phi_{\al}(\cdot,v)\|_{\La(\R)}$ is finite and Lipschitz continuous on $(1/\al,1)$ i.e. there is a finite constant $c_0>0$, such that the inequality 
\begin{equation}
\label{eq0:lcphial}
\big|\| \Phi_{\al}(\cdot,v_2)\|_{\La(\R)}-\|\Phi_{\al}(\cdot,v_1)\|_{\La(\R)}\big|\le c_0 |v_2-v_1|,
\end{equation}
holds, for all $v_1,v_2\in (1/\al,1)$. The continuity of this function entails that
\begin{equation}\label{h:maxphial}
c_1 =\max_{v\in [\underline{H},\overline{H}]} \| \Phi_{\al}(\cdot,v)\|_{\La(\R)} <+\infty;
\end{equation}
on the other hand, one has, 
\begin{equation}\label{h:minphial}
c_2 =\min_{v\in [\underline{H},\overline{H}]} \| \Phi_{\al}(\cdot,v)\|_{\La(\R)} >0.
\end{equation}
Recall that the compact interval $[\underline{H},\overline{H}]\subset (1/\al,1)$ is the range $H\big([0,1]\big)$ of the function~$H(\cdot)$.
\end{Lem}

\begin{proof}[Proof of Lemma~\ref{rem:minmaxphial}] The finiteness on $(1/\al,1)$ of the function $v \mapsto \| \Phi_{\al}(\cdot,v)\|_{\La(\R)}$ is a straightforward consequence of (\ref{eq1:locaphial}) and (\ref{eq2:locaphial}). Let us show that it is Lipschitz continuous on this interval. Assume that $v_1,v_2\in (1/\al,1)$ are arbitrary and satisfy $v_1< v_2$. It follows from the triangle inequality, (\ref{h:def:phial}) and (\ref{h:eq:pospart}) that 
\begin{eqnarray}
\label{eq1:lcphial}
&& \big|\| \Phi_{\al}(\cdot,v_2)\|_{\La(\R)}-\|\Phi_{\al}(\cdot,v_1)\|_{\La(\R)}\big|\nonumber\\
&&\le  \| \Phi_{\al}(\cdot,v_2)-\Phi_{\al}(\cdot,v_1)\|_{\La(\R)}=\bigg (\int_{-\infty}^L\Big|\sum_{l=0}^L a_l (l-u)_+^{v_2-1/\al}-\sum_{l=0}^L a_l (l-u)_+^{v_1-1/\al}\Big|^\al \, du\bigg)^{1/\al}\nonumber\\
&& \le \sum_{l=0}^L |a_l| \bigg (\int_{-2L}^L\Big|(l-u)_+^{v_2-1/\al}-(l-u)_+^{v_1-1/\al}\Big|^\al \, du\bigg)^{1/\al}\nonumber\\
&&\hspace{2cm}+\bigg (\int_{2L}^{+\infty}\Big|\sum_{l=0}^L a_l (l+x)^{v_2-1/\al}-\sum_{l=0}^L a_l (l+x)^{v_1-1/\al}\Big|^\al \, dx\bigg)^{1/\al}.
\end{eqnarray}
Let us show that 
\begin{equation}
\label{eq1bis:lcphial}
\sum_{l=0}^L |a_l| \bigg (\int_{-2L}^L\Big|(l-u)_+^{v_2-1/\al}-(l-u)_+^{v_1-1/\al}\Big|^\al \, du\bigg)^{1/\al}\le c_0 ' (v_2-v_1),
\end{equation}
where $c_0 '>0$ is a finite constant not depending on $v_1$ and $v_2$. Assume that $l\in\{0,\ldots, L\}$ is arbitrary and fixed, in view of (\ref{h:eq:pospart}) and the inequalities $0<\al v_1-1<1$, one has
\begin{eqnarray}
\label{eq2bis:lcphial}
&& \int_{-2L}^L\Big|(l-u)_+^{v_2-1/\al}-(l-u)_+^{v_1-1/\al}\Big|^\al \, du=\int_{-2L}^l\Big|(l-u)^{v_2-1/\al}-(l-u)^{v_1-1/\al}\Big|^\al \, du \nonumber\\
&&=\int_{-2L}^l (l-u)^{\al v_1-1}\Big|(l-u)^{v_2-v_1}-1\Big|^\al \, du\le3L \int_{-2L}^l \Big |\exp\Big((v_2-v_1)\log(l-u)\Big)-1\Big|^\al \, du\nonumber\\
&&\le 3c_3 L \Big(\int_{-2L}^l \big|\log(l-u)\big|^\al \,du\Big) (v_2-v_1)^\al ,
\end{eqnarray}
where the constant
$
c_3=\sup\Big\{\frac{|\exp(x)-1|}{|x|}:\,\, x\in ]-\infty,2^{-1}\log(3L)]\setminus\{0\}\Big\}<+\infty.
$
Next, using (\ref{eq2bis:lcphial}) and the fact that 
$
\int_{-2L}^l \big|\log(l-u)\big|^\al \,du<+\infty,
$
one gets (\ref{eq1bis:lcphial}). In order to show that 
\begin{equation}
\label{eq2quat:lcphial}
\bigg (\int_{2L}^{+\infty}\Big|\sum_{l=0}^L a_l (l+x)^{v_2-1/\al}-\sum_{l=0}^L a_l (l+x)^{v_1-1/\al}\Big|^\al \, dx\bigg)^{1/\al}\le c_0 '' (v_2-v_1),
\end{equation}
where $c_0 ''>0$ is a finite constant not depending on $v_1$ and $v_2$, it is enough to prove that there is a finite constant $c_4>0$, only depending on $L$, such that, for all $x\in (2L,+\infty)$, one has
\begin{equation} 
\label{eq5:lcphial}
\Big|\sum_{l=0}^L a_l (l+x)^{v_2-1/\al}-\sum_{l=0}^L a_l (l+x)^{v_1-1/\al}\Big|\le c_4 (v_2-v_1) \,x^{1/2-L}\log(x).
\end{equation} 
Applying the mean-value Theorem to the function $v\mapsto \sum_{l=0}^L a_l (l+x)^{v-1/\al}$ on the interval $[v_1,v_2]$, it follows that one has for some $w\in (v_1,v_2)\subset (1/\al,1)$,
\begin{eqnarray}
\label{eq3:lcphial}
&& \sum_{l=0}^L a_l (l+x)^{v_2-1/\al}-\sum_{l=0}^L a_l (l+x)^{v_1-1/\al}=(v_2-v_1)\sum_{l=0}^L a_l (l+x)^{w-1/\al}\log(l+x) \nonumber\\
&& =(v_2-v_1) \Big(\Phi_\al (-x,w)\log(x)+x^{w-1/\al}\sum_{l=0}^L a_l \big (1+l\, x^{-1}\big)^{w-1/\al}\log\big(1+l\, x^{-1}\big)\Big),\nonumber\\
\end{eqnarray}
where the last equality results from (\ref{h:def:phial}) and easy computations. Moreover, noticing that $z=x^{-1}$ belongs to $\big (0, 2^{-1}L^{-1}\big)$ and applying Lemma~\ref{lemj:taylor} in the case where $y_0=1$, $v_0=2$ and $f(z,v)=\big (1+z\big)^{v-1/\al}\log\big(1+z\big)$, one gets that, for all $v\in (1/\al,1)$,
\begin{equation} 
\label{eq4:lcphial}
\Big|\sum_{l=0}^L a_l \big (1+l\, x^{-1}\big)^{v-1/\al}\log\big(1+l\, x^{-1}\big)\Big|\le c_5 \,x^{-L},
\end{equation}
where $c_5>0$ is a constant not depending on $v$ and $x$.
Next, putting together (\ref{eq3:lcphial}), (\ref{eq2:locaphial}) and (\ref{eq4:lcphial}) in which one takes $v=w$, it follows that (\ref{eq5:lcphial}) holds. Finally, combining (\ref{eq1:lcphial}) with (\ref{eq1bis:lcphial}) and (\ref{eq2quat:lcphial}), one obtains (\ref{eq0:lcphial}).

Let now show that (\ref{h:minphial}) holds. The continuity of the function $v\mapsto \|\Phi_{\al}(\cdot,v)\|_{\La(\R)}$, and the compactness of the interval $[\underline{H},\overline{H}]$, imply that there exists $v_0 \in [\underline{H},\overline{H}]$ such that $c_2= \|\Phi_{\al}(\cdot,v_0)\|_{\La(\R)}$. Suppose ad absurdum that $c_2=0$, then one has, for any $u\in\R$, $\Phi_{\al}(u,v_0)=0$; which, in view of (\ref{h:def:phial}), means that
$$
\sum_{l=0}^L a_l (l-u)_+^{v_0-1/\al}=0, \quad\mbox{for all $u\in\R$}. 
$$
This is impossible; the latter equality cannot hold for instance when $u=L-1$.
\end{proof}

We are now in a position to prove Lemma \ref{h:lem:borne:scaletdjk}.
\begin{proof}[Proof of Lemma \ref{h:lem:borne:scaletdjk}]
Using (\ref{h:tdjk2}), (\ref{eqa:scparamS}), the change of variable $u=Ns-k$, and (\ref{eq1:locaphial}), one gets that,
\begin{equation}
\label{eq:eg-scaletdjk} 
\|\widetilde{d}_{N,k}\|_{\al}^{\al} = N^{-\al H(k/N)+1} \int_{\R} \big|\Phi_{\al}(N s-k,H(k/N))\big|^{\al} ds = N^{-\al H(k/N)} \big\| \Phi_{\al}(\cdot,H(k/N))\big\|_{\La(\R)}^{\al};
\end{equation}
thus, it turns out that (\ref{h:borne:scaletdjk}) is a consequence of (\ref{h:maxphial}) and (\ref{h:minphial}).
\end{proof}

The following remark is a straightforward consequence of (\ref{h:definition:Vjtilde:lfgn1}), (\ref{eqa:equivmoS}) and Lemma~\ref{h:lem:borne:scaletdjk}.

\begin{Rem}
\label{rem:encadEVT}
There are two constants $0<c_2\le c_1$, depending on $\be$ but not on $I$, such that for every integer $N\ge (L+1)\la(I)^{-1}$, one has 
\begin{equation}
\label{eq1:encadEVT}
c_2|\nu_N (I)|^{-1}\sum_{k\in\nu_N (I)}  N^{-\be H(k/N)}\le \ESP \big(\wt{V}_N^{\be} (I) \big)\le c_1|\nu_N (I)|^{-1}\sum_{k\in\nu_N (I)}  N^{-\be H(k/N)}.
\end{equation}
\end{Rem}
The expectation $\ESP \big(\wt{V}_N^{\be} (I) \big)$ can also be bounded in terms of $N^{-\be \underline{H}(I)}$.
\begin{Rem}
\label{rem:encadEVT2}
There are two constants $0<c'_2\le c_1$, depending on $\be$ and $L$ but not on $I$, such that for every integer $N\ge 4$, satisfying (\ref{eq1:encadEVT2}), one has 
\begin{equation}
\label{eq2:encadEVT2}
c'_2 \,\frac{N^{-\be \underline{H}(I)}}{(\log N)^{2}} \le \ESP \big(\wt{V}_N^{\be} (I) \big)\le c_1 N^{-\be \underline{H}(I)}.
\end{equation}
\end{Rem}

\begin{proof}[Proof of Remark~\ref{rem:encadEVT2}] First notice that (\ref{h:def1:HI}) and the second inequality in (\ref{eq1:encadEVT}) clearly imply that the second inequality in (\ref{eq2:encadEVT2}) holds. So from now, we focus on the proof of the first inequality in (\ref{eq2:encadEVT2}). Let $\mu\in I$ be such that
\begin{equation}
\label{eq3:encadEVT2}
H(\mu)=\underline{H}(I)=\min_{t\in I} H(t),
\end{equation}
and let 
\begin{eqnarray}
\label{eq4bis:encadEVT2}
&&\check{\nu}_N(I,\mu)=\bigg\{k\in\nu_N (I) : |\mu-k/N|\le \frac{2^{-1}\la (I)}{(\log N)^{2}}\bigg\}\\
&&= \bigg\{ k\in\{0,\ldots N-L\} : k/N\in I\cap \Big[\mu-\frac{2^{-1}\la (I)}{(\log N)^{2}},\mu+\frac{2^{-1}\la (I)}{(\log N)^{2}}\Big]\bigg\}. \nonumber
\end{eqnarray}
Observe that (\ref{m:cond:holder}), (\ref{eq3:encadEVT2}), (\ref{eq4bis:encadEVT2}), the inequality $\la (I)\le 1$ and the inequalities~\footnote{Those inequalities follow from (\ref{m:cond:pholder}) and the inclusions $H\big ([0,1])\subset (1/\al,1)\subset [1/2,1)$.} $\rho_H>1/\al>1/2$, entail that for all $k\in \check{\nu}_N(I,\mu)$,
\begin{eqnarray}
\label{eq4ter:encadEVT2}
&& N^{-\be H(k/N)}=\exp\big(-\be H(k/N)\log N\big)=\exp\big(\be (H(\mu)-H(k/N))\log N\big) N^{-\be \underline{H}(I)}\nonumber\\
&& \ge \exp\big(-\be |H(\mu)-H(k/N)|\log N\big) N^{-\be \underline{H}(I)}\ge\exp\big(-c_3\be (\log N)^{1-2\rho_H}\big) N^{-\be \underline{H}(I)}\nonumber\\
&& \ge c_4 N^{-\be \underline{H}(I)},
\end{eqnarray}
where $c_3$ denotes the constant $c$ in (\ref{m:cond:holder}), and $c_4=\exp(-c_3\be)$.

Let us now provide a convenient lower bound for $|\check{\nu}_N(I,\mu)|$ the cardinality of $\check{\nu}_N(I,\mu)$.
One can derive from the last equality in (\ref{eq4bis:encadEVT2}) that
\begin{equation}
\label{eq4:encadEVT2}
|\check{\nu}_N(I,\mu)|> N\la\bigg ( I\cap\Big[\mu-\frac{2^{-1}\la (I)}{(\log N)^{2}},\mu+\frac{2^{-1}\la (I)}{(\log N)^{2}}\Big] \bigg)-L-1.
\end{equation}
Next, observing that one has 
$$
\Big[\mu-\frac{2^{-1}\la (I)}{(\log N)^{2}},\mu\Big]\subseteq  I\cap\Big[\mu-\frac{2^{-1}\la (I)}{(\log N)^{2}},\mu+\frac{2^{-1}\la (I)}{(\log N)^{2}}\Big] 
$$
or 
$$
\Big[\mu,\mu+\frac{2^{-1}\la (I)}{(\log N)^{2}}\Big]\subseteq  I\cap\Big[\mu-\frac{2^{-1}\la (I)}{(\log N)^{2}},\mu+\frac{2^{-1}\la (I)}{(\log N)^{2}}\Big];
$$
then, it follows from (\ref{eq1:encadEVT2}) that
$$
\la\bigg ( I\cap\Big[\mu-\frac{2^{-1}\la (I)}{(\log N)^{2}},\mu+\frac{2^{-1}\la (I)}{(\log N)^{2}}\Big] \bigg)\ge \frac{2^{-1}\la (I)}{(\log N)^{2}}\ge \frac{4^{-1}\la (I)}{(\log N)^{2}}+\frac{L+1}{N}.
$$
Thus (\ref{eq4:encadEVT2}) implies that 
\begin{equation}
\label{eq5:encadEVT2}
|\check{\nu}_N(I,\mu)|>\frac{N\la (I)}{(2\log N)^{2}}.
\end{equation}
Finally, putting together the first inequality in (\ref{eq1:encadEVT}), the inclusion $\check{\nu}_N(I,\mu)\subseteq \nu_N (I)$, (\ref{eq4ter:encadEVT2}), (\ref{eq5:encadEVT2}) and the second inequality in (\ref{card:nuI}), one gets that 
$$
 \ESP \big(\wt{V}_N^{\be} (I) \big)\ge c_2\frac{1}{|\nu_N (I)|}\sum_{k\in\check{\nu}_N(I,\mu)}  N^{-\be H(k/N)}
\ge c_5\, \frac{|\check{\nu}_N(I,\mu)|}{|\nu_N (I)|}\,N^{-\be \underline{H}(I)}\ge c'_2 \,\frac{N^{-\be \underline{H}(I)}}{(\log N)^{2}} ,
$$
where $c_5=c_2 c_4$ and $c'_2=3c_5/14$.
\end{proof}

\begin{Lem}\label{lem:maj-scaletdNk2}
There exists a constant $c>0$, not depending on $\de\in (0,1)$, such that, for every integers $N\ge L$ and $0\le k \le N-L$, one has
\begin{equation}\label{eq1:maj-scaletdNk2}
\|\widetilde{d}_{N,k}^{\,2,\de} \|_{\al} \le c\, N^{-(1-\de)H(k/N)-\de L}.
\end{equation}
\end{Lem}

\begin{proof}[Proof of Lemma~\ref{lem:maj-scaletdNk2}] Using (\ref{eq:def-dtjN2}), (\ref{eqa:scparamS}) and the change of  variable $u=Ns-k$, we obtain
\begin{eqnarray}
\label{eq2:maj-scaletdNk2}
\|\widetilde{d}_{N,k}^{\,2,\de} \|_{\al}^\al &=& N^{-\al H(k/N)+1} \int_{-\infty}^{N^{-1}(k-e_N+L)} \big|\Phi_{\al}(N s-k,H(k/N))\big|^{\al} ds  \nonumber\\
&=& N^{-\al H(k/N)} \int_{-\infty}^{L-e_N} \big|\Phi_{\al} (u,H(k/N))\big|^{\al} du.
\end{eqnarray}
Next, denoting by $c_1$ the constant $c$ in (\ref{eq2:locaphial})  then, it follows from (\ref{eq2:maj-scaletdNk2}) and (\ref{eq2:locaphial})  that 
\begin{equation}
\label{eq5:maj-scaletdNk2}
\|\widetilde{d}_{N,k}^{\,2,\de} \|_{\al}^\al \le   c_1 ^\al \, N^{-\al H(k/N)} \int_{-\infty}^{L-e_N} 
\big(1+L+|u|\big)^{\al H(k/N)-\al L-1} du.
\end{equation}
Observe that when $e_N<L$ (i.e. $N^\de<L$ see (\ref{m:eq:add1})) then (\ref{eq5:maj-scaletdNk2}) entails that
\begin{equation}
\label{eq6:maj-scaletdNk2}
\|\widetilde{d}_{N,k}^{\,2,\de} \|_{\al} \le c_2\, N^{-(1-\de)H(k/N)-\de L}
\end{equation}
where the finite constant  
$
c_2=c_1 L^{L} \Big (\int_{-\infty}^{L} 
\big(1+L+|u|\big)^{-\al( L-1)-1} du\Big)^{1/\al}.
$
On the other hand, when $e_N\ge L$, (\ref{eq5:maj-scaletdNk2}) and (\ref{m:eq:add1}) imply that
\begin{eqnarray}
\label{eq3:maj-scaletdNk2}
\|\widetilde{d}_{N,k}^{\,2,\de} \|_{\al}^\al &\le &  c_1 ^\al \, N^{-\al H(k/N)} \int_{-\infty}^{L-e_N} 
\big(1+L-u\big)^{\al H(k/N)-\al L-1} du\nonumber\\
&\le & c_1 ^\al \, N^{-\al H(k/N)} \big (\al L-\al H(k/N)\big)^{-1}\big (e_N+1\big)^{\al (H(k/N)-L)} \nonumber\\
&\le &  c_1 ^\al  N^{-\al(1-\de)H(k/N)-\al\de L}.
\end{eqnarray}

Finally, taking  $c=c_2+c_1$, in view of (\ref{eq3:maj-scaletdNk2}) and (\ref{eq6:maj-scaletdNk2}), one gets the lemma.
\end{proof}

The following remark is a straightforward consequence of (\ref{eqa:equivmoS}), Lemma~\ref{lem:maj-scaletdNk2} and~(\ref{h:def1:HI}).

\begin{Rem}
\label{rem:dVV1td}
Let $\wt{V}_N^{\be,2,\de} (I)$ be as in (\ref{req:Vjtilde2}).
There is a constant $c>0$, depending on $\be$ but not on $I$ and $\de$, such that for every integer $N\ge (L+1)\la(I)^{-1}$, one has 
\begin{equation}
\label{eq1:dVV1td}
\ESP \big(\wt{V}_N^{\be,2,\de} (I)\big)\le  c\, N^{-(1-\de)\be\underline{H}(I)-\de\be L}.
\end{equation}
\end{Rem}

The following remark is a straightforward consequence of  (\ref{h:definition:Vjtilde:lfgn1}), (\ref{req:Vjtilde1}), (\ref{eq:expr-dtjN}), (\ref{trian-ineg-be}), (\ref{req:Vjtilde2}) and (\ref{eq1:dVV1td}).

\begin{Rem}
\label{rem-bis:dVV1td}
For every integer $N\ge (L+1)\la(I)^{-1}$, one has 
\begin{equation}
\label{eq1-bis:dVV1td}
\Big|\ESP \big(\wt{V}_N^{\be} (I)\big)-\ESP \big(\wt{V}_N^{\be,1,\de} (I)\big)\Big|\le \ESP \big(\wt{V}_N^{\be,2,\de} (I)\big)\le  c\, N^{-(1-\de)\be\underline{H}(I)-\de\be L},
\end{equation}
where $c$ is the same constant as in (\ref{eq1:dVV1td}).
\end{Rem}

We are now in a position to prove Lemmas~\ref{lem:maj-prob2} and \ref{lem:maj-prob3}.
\begin{proof}[Proof of Lemma~\ref{lem:maj-prob2}]
First, notice that using Markov inequality one gets that
\begin{equation}
\label{eq1:newlem1}
\PR\Bigg( \Big| \frac{\wt{V}_N^{\be,1,\de}(I)}{\ESP \big(\wt{V}_N^{\be} (I) \big)} -1  \Big| > \eta \Bigg)\le \eta^{-4}\,\frac{\ESP\bigg(\Big (\wt{V}_N^{\be,1,\de}(I)-\ESP \big(\wt{V}_N^{\be} (I) \big)\Big)^4\bigg)}{\Big (\ESP \big(\wt{V}_N^{\be} (I) \big)\Big)^4}.
\end{equation}
Thanks to the first inequality in (\ref{eq2:encadEVT2}), one can conveniently bound from below the denominator in the right-hand side of (\ref{eq1:newlem1}). From now on, our goal is to find a convenient upper bound of the numerator in it. The convexity of the function $x\mapsto x^4$ and (\ref{eq1-bis:dVV1td}) entail that
\begin{eqnarray}
\label{eq1-bis:newlem1}
&& \ESP\bigg(\Big (\wt{V}_N^{\be,1,\de}(I)-\ESP \big(\wt{V}_N^{\be} (I) \big)\Big)^4\bigg) \nonumber\\
&& \le  8\,\ESP\bigg(\Big (\wt{V}_N^{\be,1,\de}(I)-\ESP \big(\wt{V}_N^{\be,1,\de} (I) \big)\Big)^4\bigg) +8\Big (\ESP \big(\wt{V}_N^{\be} (I)\big)-\ESP \big(\wt{V}_N^{\be,1,\de} (I)\big)\Big)^4\nonumber\\
&& \le 8\,\ESP\bigg(\Big (\wt{V}_N^{\be,1,\de}(I)-\ESP \big(\wt{V}_N^{\be,1,\de} (I) \big)\Big)^4\bigg) +c_1 N^{-4(1-\de)\be\underline{H}(I)-4\de\be L},
\end{eqnarray}
where $c_1>0$ is a constant not depending on $I$, $\de$ and $N$. Next observe that, in view of (\ref{req:Vjtilde1}), one has 
\begin{equation}
\label{eq1-ter:newlem1}
\ESP\bigg(\Big (\wt{V}_N^{\be,1,\de}(I)-\ESP \big(\wt{V}_N^{\be,1,\de} (I) \big)\Big)^4\bigg)=|\nu_N(I)|^{-4}\,\ESP\bigg(\Big ( \sum_{k\in\nu_N (I)} \Delta_{N,k}\Big)^4\bigg),
\end{equation}
where the $\Delta_{N,k}$'s are the centered random variables defined as
\begin{equation}
\label{eq2:newlem1}
\Delta_{N,k}=|\wt{d}_{N,k}^{\,1,\de}|^\be-\ESP\big (|\wt{d}_{N,k}^{\,1,\de}|^\be\big);
\end{equation}
moreover (\ref{eq3:newlem1}) and the convexity of the function $x\mapsto x^4$, imply that
\begin{equation}
\label{eq4:newlem1}
\ESP\bigg(\Big ( \sum_{k\in\nu_N (I)} \Delta_{N,k}\Big)^4\bigg)\le e_N ^3 \sum_{r=0}^{e_N-1}\ESP\bigg(\Big ( \sum_{k\in\jnr} \Delta_{N,k}\Big)^4\bigg),
\end{equation}
with the convention that $\sum_{k\in\jnr} \ldots =0$ when $\jnr$ is the empty set.
Next, using the crucial fact that, for each fixed $r\in\{0,\ldots,e_N-1\}$, $\{\Delta_{N,k}:k\in \jnr\}$ is 
a finite sequence of independent centered random variables, one gets
\begin{eqnarray}
\label{eq5:newlem1}
&& \ESP\bigg(\Big ( \sum_{k\in\jnr} \Delta_{N,k}\Big)^4\bigg)\\
&& = \sum_{k\in\jnr} \ESP\Big(\big(\Delta_{N,k} \big)^4\Big)+\sum_{(k',k'')\in\jnr^2,\, k'\ne k''} \ESP\Big(\big (\Delta_{N,k'}\big)^2\Big)\,\ESP\Big(\big (\Delta_{N,k''}\big)^2\Big).\nonumber
\end{eqnarray}
The convexity of the function $x\mapsto x^4$, (\ref{eq2:newlem1}), (\ref{eqa:equivmoS}), the inequality~\footnote{This inequality follows from (\ref{eqa:scparamS}), (\ref{h:tdjk2}) and (\ref{eq:def-dtjN1}).} $\|\wt{d}_{N,k}^{\,1,\de}\|_\al \le\|\wt{d}_{N,k}\|_\al$, the second inequality in (\ref{h:borne:scaletdjk}), and (\ref{h:def1:HI}), imply that 
\begin{eqnarray}
\label{eq6:newlem1}
\ESP\Big(\big(\Delta_{N,k} \big)^4\Big) &\le & 8\,\ESP\big (|\wt{d}_{N,k}^{\,1,\de}|^{4\be}\big)+8\Big(\ESP\big (|\wt{d}_{N,k}^{\,1,\de}|^\be\big)\Big)^4\nonumber\\
&\le & c_2 \|\wt{d}_{N,k}^{\,1,\de}\|_\al ^{4\be} \le c_{3} N^{- 4\be H(k/N)}\le  c_{3} N^{- 4\be \underline{H}(I)};
\end{eqnarray}
moreover, using the convexity of the function $x\mapsto x^2$ and similar arguments, one gets that
\begin{equation}
\label{eq7:newlem1}
\ESP\Big(\big(\Delta_{N,k} \big)^2\Big)\le  c_{4} N^{- 2\be \underline{H}(I)}.
\end{equation}
Observe that the constants $c_2$, $c_3$ and $c_4$ do not depend on $I$, $\de$, $N$, $k$ and $r$. Next setting $c_5=\max(c_3,c_{4}^2)$, then (\ref{eq5:newlem1}), (\ref{eq6:newlem1}) and (\ref{eq7:newlem1}), entail that 
\begin{equation}
\label{eq8:newlem1}
\ESP\bigg(\Big ( \sum_{k\in\jnr} \Delta_{N,k}\Big)^4\bigg)\le c_5 |\jnr|^2 N^{- 4\be \underline{H}(I)},
\end{equation}
where $|\jnr|$ denotes the cardinality of the set $\jnr$; it easily follows from (\ref{eq:INr}) that, for each $r\in \{0,\ldots, e_N-1\}$, 
\begin{equation}
\label{eq9:newlem1}
|\jnr|\le \frac{|\nu_N (I)|}{e_N}+1.
\end{equation}
Next combining (\ref{eq4:newlem1}) with (\ref{eq8:newlem1}) and (\ref{eq9:newlem1}), one obtains that
\begin{equation}
\label{eq10:newlem1}
\ESP\bigg(\Big ( \sum_{k\in\nu_N (I)} \Delta_{N,k}\Big)^4\bigg)\le c_6 \big( e_N ^2 |\nu_N (I)|^2 +e_N^4 \big)N^{- 4\be \underline{H}(I)},
\end{equation}
where $c_6>0$ is a constant not depending on $I$, $\de$ and $N$. 
Next putting together (\ref{eq1-bis:newlem1}), (\ref{eq1-ter:newlem1}) and (\ref{eq10:newlem1}), it follows that
\begin{equation}
\label{eq12:newlem1}
\ESP\bigg(\Big (\wt{V}_N^{\be,1}(I)-\ESP \big(\wt{V}_N^{\be} (I) \big)\Big)^4\bigg) \le c_7 \Big (\frac{e_N ^2}{|\nu_N (I)|^2}+\frac{e_N ^4}{|\nu_N (I)|^4}+N^{-4\de\be (L-\underline{H}(I))}\Big) N^{- 4\be \underline{H}(I)},
\end{equation}
where $c_7>0$ is a constant not depending on $I$, $\de$ and $N$. Next, observe that (\ref{eq1:encadEVT2}) clearly implies that 
$
N\ge 2(L+1)\la (I)^{-1}
$
and consequently that (\ref{card:nuI}) holds.
Combining the first inequality in (\ref{card:nuI}) with (\ref{eq12:newlem1}) and (\ref{m:eq:add1}), one obtains that 
\begin{equation}
\label{eq13:newlem1}
\ESP\bigg(\Big (\wt{V}_N^{\be,1}(I)-\ESP \big(\wt{V}_N^{\be} (I) \big)\Big)^4\bigg) < 8 c_7 \bigg (\frac{N^{-2(1-\de)}}{\la(I)^{2}}+\frac{N^{-4(1-\de)}}{\la(I)^{4}}+N^{-4\de\be (L-\underline{H}(I))}\bigg) N^{- 4\be \underline{H}(I)}.
\end{equation}
Finally, (\ref{eq1:maj-prob2}) results from (\ref{eq1:newlem1}), (\ref{eq13:newlem1}) and the first inequality in (\ref{eq2:encadEVT2}).
\end{proof}

%Before proving Lemma~\ref{lem:maj-prob1}, we mention that by combining it with Lemma~\ref{h:lem:ecartdjktdjk} and Remark~\ref{rem:encadEVT}, one can easily obtain the following 
%result.

%\begin{Lem}
%\label{lem:maj-prob1bis}
%For any fixed integer $L\ge 2$ and real number $\be\in (0,1/4]$, there is a constant $c>0$ satisfying the following property: for all compact interval $I\subseteq [0,1]$ with non-empty interior, %and for each real numbers $\de\in (0,1/2]$ and $\eta>0$, the inequality 
%\begin{eqnarray}
%\label{eq1:maj-prob1bis}
%\PR\Bigg( \Big| \frac{V_N^{\be}(I)}{\ESP \big(\wt{V}_N^{\be} (I) \big)} -1  \Big| > \eta \Bigg)&\le & c\, \eta^{-1}\bigg ( N^{-\de\be (L-\underline{H}(I))}(\log N)^{2}+N^{-\beta %(\rho_H-\overline{H})}\bigg)\\\
%&& \hspace{0.4cm}+c\, \eta^{-4}\bigg (\frac{N^{-2(1-\de)}}{\la(I)^{2}}+\frac{N^{-4(1-\de)}}{\la(I)^{4}}+N^{-4\de\be (L-\underline{H}(I))}\bigg)(\log N)^{8}\nonumber
%\end{eqnarray}
%holds, for every integer $N\ge 4$ which satisfies (\ref{eq1:encadEVT2}). We recall that $V_N^{\be}(I)$, $\rho_H$ and $\overline{H}$ have been defined in (\ref{h:def:VJ}), (\ref{m:cond:holder}), %and 
%(\ref{m:cond:pholder}).
%\end{Lem}

\begin{proof}[Proof of Lemma~\ref{lem:maj-prob3}] Using Markov inequality one gets that
$$
\PR\Bigg( \frac{\wt{V}_N^{\be,2,\de}(I)}{\ESP \big(\wt{V}_N^{\be} (I) \big)}  > \eta \Bigg)\le \eta^{-1}\,\frac{\ESP\big( \wt{V}_N^{\be,2,\de}(I)\big ) }{\ESP \big (\wt{V}_N^{\be} (I) \big)},
$$
then combining (\ref{eq1:dVV1td}) with the first inequality in (\ref{eq2:encadEVT2}), it follows that (\ref{eq1:maj-prob3}) is satisfied.
\end{proof}

\section{Proofs of Theorems~\ref{thm:main1bis} and \ref{thm:main1}}
\label{sec:proofsmainbismain1}

First, we focus on the proof Theorem~\ref{thm:main1}. We need some preliminary results. Let us recall that the empirical mean $\wt{V}_N^{\be}(I)$ is defined through (\ref{h:definition:Vjtilde:lfgn1}). The following lemma is, in some sense, a strong law of large numbers for $\wt{V}_N^{\be}(I)$ with a rate of convergence almost surely bounded
by a power function.
\begin{Lem}
\label{lem:BorCan1}
Assume that the real number $\be\in (0,1/4]$ and the integer $L$ are arbitrary and such that the inequality (\ref{eq1:BorCan1}) holds. Let $I\subseteq [0,1]$ be an arbitrary compact interval 
with non-empty interior and let $\ga$ be an arbitrary positive real number satisfying
\begin{equation}
\label{eq2:BorCan1}
\ga<\frac{\be(L-\underline{H}(I))-2}{4\be(L-\underline{H}(I))+2},
\end{equation}
where $\underline{H}(I)$ is defined through (\ref{h:def1:HI}). Then, there exists a positive finite random variable $C$, such that the inequality 
\begin{equation}
\label{eq3:BorCan1}
\Big| \frac{\wt{V}_{j N}^{\be}(I)}{\ESP \big(\wt{V}_{jN}^{\be} (I) \big)} -1  \Big| \le C N^{-\ga},
\end{equation}
holds almost surely, for any integer $N\ge (L+1)\la(I)^{-1}$ and $j\in\{1,2\}$. 
%Moreover, denoting by $p$ an arbitrary fixed positive real number, one has 
%\begin{equation}
%\label{eq3bis:BorCan1}
%\ESP\Bigg(\bigg| \Ch\bigg(\frac{V_N^{\be}(I)}{\ESP \big(\wt{V}_N^{\be} (I) \big)}\bigg) -1  \bigg|^p\Bigg) \le c(p) N^{-p\ga},
%\end{equation}
%where $\Ch(\cdot)$ is as in (\ref{eq:troncC}) and $c(p)$ a finite constant depending on $p$ but not on $N$.
\end{Lem}

Before giving the proof of Lemma~\ref{lem:BorCan1}, let us point out that Lemma~\ref{lem:maj-prob1} is the main ingredient of this proof.

\begin{proof}[Proof of Lemma~\ref{lem:BorCan1}] Let us first assume that $N$ is big enough, so that (\ref{eq1:encadEVT2}) is satisfied. Then, taking in (\ref{eq1:maj-prob1}), $\eta=N^{-\ga}$ and using the inequality $\log(2N)\le 2\log N$, one gets that 
\begin{eqnarray}
\label{eq4:BorCan1}
&& \PR\Bigg( \Big| \frac{\wt{V}_{jN}^{\be}(I)}{\ESP \big(\wt{V}_{jN}^{\be} (I) \big)} -1  \Big| > N^{-\ga} \Bigg)\\
&& \le  c_1\, N^{\ga-\de\be (L-\underline{H}(I))}(\log N)^{2}
+c_1\, \bigg (\frac{N^{-2(1-\de-2\ga)}}{\la(I)^{2}}+\frac{N^{-4(1-\de-\ga)}}{\la(I)^{4}}+N^{4\ga-4\de\be (L-\underline{H}(I))}\bigg)(\log N)^{8},\nonumber
\end{eqnarray}
where $\de\in (0,1)$ is arbitrary and $c_1>0$ is a constant not depending on $N$, $\ga$ and $\de$. Next, one sets 
$$
\de=\frac{5}{4\be(L-\underline{H}(I))+2};
$$
notice that the latter quantity belongs to the interval $(0,1/2]$ because of (\ref{eq1:BorCan1}). Standard computations, relying on (\ref{eq2:BorCan1}), allow to derive, for such a choice of $\de$, that 
\begin{equation}
\label{eq5:BorCan1}
2(1-\de-2\ga)>1 \quad \mbox{and} \quad\de\be (L-\underline{H}(I))-\ga>1.
\end{equation} 
Finally, in view of (\ref{eq4:BorCan1}) and (\ref{eq5:BorCan1}), one gets that 
$$
\sum_{N\ge (L+1)\la(I)^{-1}}^{+\infty}  \PR\Bigg( \Big| \frac{\wt{V}_{jN}^{\be}(I)}{\ESP \big(\wt{V}_{jN}^{\be} (I) \big)} -1  \Big| > N^{-\ga} \Bigg)<+\infty,
$$
then, by using the Borel-Cantelli lemma, one obtains (\ref{eq3:BorCan1}).
\end{proof}

The following remark is a straightforward consequence of Lemma~\ref{lem:BorCan1}.
\begin{Rem}
\label{Rem:BorCan1}
Assume that $\be$, $L$ and $I$ are as in Lemma~\ref{lem:BorCan1}.
Then, it follows from this lemma that there exists a positive finite random variable $C$, such that the inequality 
$$
\Big| \frac{\wt{V}_{jN}^{\be}(I)}{\ESP \big(\wt{V}_{jN}^{\be} (I) \big)} -1  \Big| \le C N^{-\frac{\be(L-1)-2}{4\be(L-1)+2}}
$$
holds almost surely, for any $N\ge (L+1)\la(I)^{-1}$ and $j\in\{1,2\}$. Notice that the lemma can be used since
\begin{equation}
\label{eq2ter:BorCan1}
0<\frac{\be(L-1)-2}{4\be(L-1)+2}<\frac{\be(L-\underline{H}(I))-2}{4\be(L-\underline{H}(I))+2}.
\end{equation}
\end{Rem}

Recall that the estimator $\wh{H}(I)$ (see (\ref{eq1:main1})) is defined through the random ratio $V_N^\be (I)/ V_{2N}^\be (I)$. In view of Lemma~\ref{lem:BorCan1}, Lemma~\ref{h:lem:ecartdjktdjk}, and Remark~\ref{rem:encadEVT2}, it turns out that, when $N$ goes to $+\infty$, the asymptotic behavior of $V_N^\be (I)/ V_{2N}^\be (I)$ is, almost surely, equivalent to that of the deterministic ratio $\ESP \big(\wt{V}_N^{\be} (I)\big)/\ESP \big(\wt{V}_{2N}^{\be} (I)\big)$. Notice that in the particular case of a LFSM, where the Hurst function $H(\cdot)$ is a constant denoted by $H$, using the self-similarity and the stationarity of increments of this process, it can be easily seen that $\ESP \big(\wt{V}_N^{\be} (I)\big)/\ESP \big(\wt{V}_{2N}^{\be} (I)\big)=2^{\be H}$.  In the general case of a LMSM, Remark~\ref{rem:ratioVNbis}, the following lemma shows that this ratio converges to $2^{\be \underline{H} (I)}$.

\begin{Lem}
\label{lem:ratio1-esp}
There are three positive constants $c$, $c'$ and $c''$  such that, for each compact interval $I\subseteq [0,1]$ with non-empty interior, and for all integer $N\ge 4$ satisfying (\ref{eq1:encadEVT2}), one has
\begin{equation}
\label{eq1:ratio1-esp}
B_N =\bigg |\frac{2^{\be\underline{H}(I)}\,\ESP \big(\wt{V}_{2N}^{\be} (I)\big)}{\ESP \big(\wt{V}_{N}^{\be}(I)\big)}-1\bigg |\le c\,T_N (\la (I),\be),
\end{equation}
\begin{equation}
\label{eq1:ratioVNbis}
\frac{\ESP \big(\wt{V}_{N}^{\be}(I)\big)}{2^{\be\underline{H}(I)}\,\ESP \big(\wt{V}_{2N}^{\be} (I)\big)}\le c'
\end{equation}
and
\begin{equation}
\label{eq2:ratioVNbis}
\bigg |\frac{\ESP \big(\wt{V}_{N}^{\be}(I)\big)}{2^{\be\underline{H}(I)}\,\ESP \big(\wt{V}_{2N}^{\be} (I)\big)}-1\bigg | \le  c''\, T_N (\la (I), \be)\,,
\end{equation}
where, for any $(\mu,\be)\in (0,1]\times (0,1/4]$,
\begin{equation}
\label{eq0:ratio1-esp}
T_N (\mu,\be)=\min\Big\{\frac{\log \log N}{\log N}, \mu^{\rho_H}\Big\}+N^{-\be\rho_H}(\log N)^3+\mu^{-1}N^{-1} (\log N)^2.
\end{equation}
\end{Lem}

In order to show that Lemma~\ref{lem:ratio1-esp} holds, one needs the following lemma.
\begin{Lem}
\label{lem:pratio1-esp}
There exists a constant $c>0$ such that, for each compact interval $I\subseteq [0,1]$ with non-empty interior, and for all integer $N\ge 4$ satisfying (\ref{eq1:encadEVT2}), one has
\begin{eqnarray}
\label{eq1:pratio1-esp}
 A_N &=&\Bigg |\frac{\sum_{k\in\nu_N (I)} \big\|\Phi_\al (\cdot,H(k/N))\big\|_{L^\al(\R)}^\be \,(2N)^{-\be H(k/N)}}{\sum_{k\in\nu_N (I)} \big\|\Phi_\al (\cdot,H(k/N))\big\|_{L^\al(\R)}^\be\,N^{-\be H(k/N)}}-\frac{\ESP \big(\wt{V}_{2N}^{\be} (I)\big)}{\ESP \big(\wt{V}_{N}^{\be}(I)\big)}\Bigg | \nonumber\\
\\
& \le & c\,N^{-\be\rho_H}(\log N)^3+c\,\la(I)^{-1}N^{-1} (\log N)^2. \nonumber
\end{eqnarray}
\end{Lem} 

\begin{proof}[Proof of Lemma~\ref{lem:pratio1-esp}] It follows from (\ref{h:definition:Vjtilde:lfgn1}), (\ref{h:tdjk2}), (\ref{eqa:equivmoS}), and (\ref{eq:eg-scaletdjk}) that
\begin{equation}
\label{eq2:pratio1-esp}
\frac{\ESP \big(\wt{V}_{2N}^{\be} (I)\big)}{\ESP \big(\wt{V}_{N}^{\be}(I)\big)}=\frac{|\nu_N (I)|}{|\nu_{2N} (I)|}\times \frac{\sum_{m\in\nu_{2N} (I)} \big\|\Phi_\al (\cdot,H(m/2N))\big\|_{L^\al(\R)}^\be\,  (2N)^{-\be H(m/2N)}}{\sum_{k\in\nu_N (I)} \big\|\Phi_\al (\cdot,H(k/N))\big\|_{L^\al(\R)}^\be\,N^{-\be H(k/N)}}\,.
\end{equation}
Thus, one has
\begin{eqnarray}
\label{eq2bis:pratio1-esp}
&& A_N\le  \Big(\sum_{k\in\nu_N (I)} \big\|\Phi_\al (\cdot,H(k/N))\big\|_{L^\al(\R)}^\be\,N^{-\be H(k/N)}\Big)^{-1}\times\\
&&\hspace{1.5cm}\Bigg(\frac{|\nu_N (I)|}{|\nu_{2N} (I)|}\,R_N+\bigg |1-\frac{2|\nu_N (I)|}{|\nu_{2N} (I)|}\bigg|\sum_{k\in\nu_N (I)} \big\|\Phi_\al (\cdot,H(k/N))\big\|_{L^\al(\R)}^\be\,(2N)^{-\be H(k/N)} \Bigg),\nonumber
\end{eqnarray}
where
\begin{eqnarray}
\label{eq2ter:pratio1-esp}
&& R_N=\Big|\sum_{m\in\nu_{2N} (I)} \big\|\Phi_\al (\cdot,H(m/2N))\big\|_{L^\al(\R)}^\be\, (2N)^{-\be H(m/2N)}\\
&& \hspace{3.5cm}-2\sum_{k\in\nu_N (I)} \big\|\Phi_\al (\cdot,H(k/N))\big\|_{L^\al(\R)}^\be\,  (2N)^{-\be H(k/N)}\Big|.\nonumber
\end{eqnarray}
Notice that (\ref{h:def:nuI}) easily implies that
\begin{equation}
\label{eq2june:pratio1-esp}
\big\{2k: k\in\nu_{N} (I)\big\}\subset \nu_{2N} (I).
\end{equation}
So, let $\nu_{2N} ^{o}(I)$ and $\nu_{2N} ^{e}(I)$ be the sets of indices defined as 
\begin{equation}
\label{eq3bis:pratio1-esp}
\nu_{2N} ^{o}(I)=\nu_{2N} (I)\setminus \big\{2k: k\in\nu_{N} (I)\big\}
\end{equation}
and 
$$
\nu_{2N} ^{e}(I)=\big\{k\in \nu_{N} (I) : \min \nu_{2N} ^{o}(I) < 2k < \max \nu_{2N} ^{o}(I)\big\}.
$$
Observe that
\begin{equation}
\label{eq3:pratio1-esp}
|\nu_N (I)\setminus \nu_{2N} ^{e}(I)|\le 2, 
\end{equation}
and 
\begin{equation}
\label{eq4:pratio1-esp}
\nu_{2N} ^{o}(I)=\big\{\zeta_N\big\}\cup \big\{2k+1: k\in\nu_{2N} ^{e}(I)\big\},
\end{equation}
where $\zeta_N=\min \nu_{2N} ^{o}(I)$. In view of (\ref{eq2ter:pratio1-esp}), (\ref{eq2june:pratio1-esp}), (\ref{eq3bis:pratio1-esp}) and (\ref{eq4:pratio1-esp}), one has 
\begin{equation}
\label{eq5:pratio1-esp}
R_N\le \sum_{p=1}^3 R_{p,N}, 
\end{equation}
where
\begin{eqnarray*}
&& R_{1,N}=\sum_{k\in \nu_{2N} ^{e}(I)} \Big|\big\|\Phi_\al \big(\cdot,H((2k+1)/2N)\big)\big\|_{L^\al(\R)}^\be\, (2N)^{-\be H((2k+1)/2N)}\nonumber\\
&&\hspace{8cm}-\big\|\Phi_\al (\cdot,H(k/N))\big\|_{L^\al(\R)}^\be\, (2N)^{-\be H(k/N)}\Big |\nonumber
\end{eqnarray*}
\begin{eqnarray*}
R_{2,N}=\sum_{k\in\nu_N (I)\setminus \nu_{2N} ^{e}(I)} \big\|\Phi_\al (\cdot,H(k/N))\big\|_{L^\al(\R)}^\be\, (2N)^{-\be H(k/N)}
\end{eqnarray*}
and 
\begin{eqnarray*}
R_{3,N}=\big\|\Phi_\al (\cdot,H(\zeta_N/2N))\big\|_{L^\al(\R)}^\be\,  (2N)^{-\be H(\zeta_N/2N)}.
\end{eqnarray*}
Using (\ref{trian-ineg-be}), (\ref{eq0:lcphial}), (\ref{m:cond:holder}), (\ref{h:maxphial}), the mean-value Theorem, (\ref{h:def1:HI}) and the inclusion $\nu_{2N} ^{e}(I)\subseteq\nu_N (I)$, one can show that
\begin{equation}
\label{eq6:pratio1-esp}
 R_{1,N}\le c_1 \,|\nu_N (I)|\, N^{-\be(\underline{H}(I)+\rho_H )}\log N;
\end{equation}
moreover, (\ref{h:maxphial}), (\ref{h:def1:HI}) and (\ref{eq3:pratio1-esp}) imply that
\begin{equation}
\label{eq7bis:pratio1-esp}
R_{2,N}+R_{3,N}\le c_2 N^{-\be\underline{H}(I)}.
\end{equation}
Combining (\ref{eq5:pratio1-esp}) with (\ref{eq6:pratio1-esp}) and (\ref{eq7bis:pratio1-esp}), one gets that
\begin{equation}
\label{eq7:pratio1-esp}
R_N\le c_3 N^{-\be\underline{H}(I)}\big(1+ |\nu_N (I)|\,N^{-\be\rho_H }\log N\big);
\end{equation}
observe that the constants $c_1$, $c_2$ and $c_3$ do not depend on $N$ and $I$. Next, putting together, (\ref{h:definition:Vjtilde:lfgn1}), (\ref{h:tdjk2}), (\ref{eqa:equivmoS}), (\ref{eq:eg-scaletdjk}), and the first inequality in (\ref{eq2:encadEVT2}), one gets that
\begin{equation}
\label{eq8:pratio1-esp}
\Big(\sum_{k\in\nu_N (I)} \big\|\Phi_\al (\cdot,H(k/N))\big\|_{L^\al(\R)}^\be\,N^{-\be H(k/N)}\Big)^{-1}\le 
c_4 \, |\nu_N (I)|^{-1} N^{\be \underline{H}(I)} (\log N)^2,
\end{equation}
where the constant $c_4$ does not depend on $N$ and $I$. One clearly has that 
\begin{equation}
\label{eq9:pratio1-esp}
\frac{|\nu_N (I)|}{|\nu_{2N} (I)|}\le 1.
\end{equation}
On the other hand, using (\ref{card-june:nuI}) and (\ref{card:nuI}), one has 
\begin{equation}
\label{eq10:pratio1-esp}
\bigg |1-\frac{2|\nu_N (I)|}{|\nu_{2N}(I)|}\bigg|\le c_5\, |\nu_N (I)|^{-1},
\end{equation}
where the constant $c_5$ does not depend on $N$ and $I$. Also observe that, in view of (\ref{h:maxphial}) and (\ref{h:def1:HI}), one has 
\begin{equation}
\label{eq11:pratio1-esp}
\sum_{k\in\nu_N (I)} \big\|\Phi_\al (\cdot,H(k/N))\big\|_{L^\al(\R)}^\be\,(2N)^{-\be H(k/N)} \le c_6\, |\nu_N (I)|\,N^{-\be \underline{H}(I)} 
\end{equation}
where the constant $c_6$ does not depend on $N$ and $I$. Finally, putting together (\ref{eq2bis:pratio1-esp}),
(\ref{eq7:pratio1-esp}), (\ref{eq8:pratio1-esp}), (\ref{eq9:pratio1-esp}), (\ref{eq10:pratio1-esp}), (\ref{eq11:pratio1-esp}) and (\ref{card:nuI}), it follows that (\ref{eq1:pratio1-esp}) holds.
\end{proof}

Now, we are in a position to show that Lemma~\ref{lem:ratio1-esp} holds.

\begin{proof}[Proof of Lemma~\ref{lem:ratio1-esp}] Let us first show that (\ref{eq1:ratio1-esp}) holds. Using the equalities in (\ref{eq1:ratio1-esp}) and (\ref{eq1:pratio1-esp}), as well as the triangle inequality, one has 
\begin{equation}
\label{eq2:ratio1-esp}
B_N\le G_N +2^{\be\underline{H}(I)}A_N\le G_N +2A_N,
\end{equation}
where 
$$
G_N =\frac{\sum_{k\in\nu_N (I)} \big\|\Phi_\al (\cdot,H(k/N))\big\|_{L^\al(\R)}^\be\,|2^{\be(\underline{H}(I)-H(k/N))}-1|\, N^{-\be H(k/N)}}{\sum_{k\in\nu_N (I)} \big\|\Phi_\al (\cdot,H(k/N))\big\|_{L^\al(\R)}^\be\,N^{-\be H(k/N)}}.
$$
Notice that, there exists a constant $c_1$, not depending on $I$ and $N$, such that 
\begin{equation}
\label{eq3:ratio1-esp}
G_N \le\frac{c_1\,\sum_{k\in\nu_N (I)} \big\|\Phi_\al (\cdot,H(k/N))\big\|_{L^\al(\R)}^\be\,|\underline{H}(I)-H(k/N)|\, N^{-\be H(k/N)}}{\sum_{k\in\nu_N (I)} \big\|\Phi_\al (\cdot,H(k/N))\big\|_{L^\al(\R)}^\be\,N^{-\be H(k/N)}};
\end{equation}
also, notice that one can derive from (\ref{eq3:encadEVT2}), (\ref{h:def:nuI}) and (\ref{m:cond:holder}) that
\begin{equation}
\label{eq4:ratio1-esp}
|\underline{H}(I)-H(k/N)|\le c_2 \la(I)^{\rho_H}, \quad\mbox{for all $k\in\nu_N (I)$,}
\end{equation} 
where $c_2$ denotes the constant $c$ in (\ref{m:cond:holder}).
Let $\underline{\nu_N} (I)$ and $\overline{\nu_N} (I)$ be the sets of indices defined as:
\begin{equation}
\label{eq5:ratio1-esp}
\underline{\nu_N} (I)=\bigg\{k\in\nu_N (I):  |\underline{H}(I)-H(k/N)|\le\frac{4\log\log N}{\be \log N}\bigg\}
\end{equation} 
and 
\begin{equation}
\label{eq5bis:ratio1-esp}
\overline{\nu_N} (I)=\bigg\{k\in\nu_N (I):  |\underline{H}(I)-H(k/N)|>\frac{4\log\log N}{\be \log N}\bigg\};
\end{equation}
One clearly has that
\begin{equation}
\label{eq5ter:ratio1-esp}
\nu_N=\underline{\nu_N} (I)\cup \overline{\nu_N} (I)\,\, \,\,\mbox{(disjoint union).}
\end{equation}
Next setting $c_3=\max(c_2, 4/\be)$, then it follows from(\ref{eq4:ratio1-esp}) and (\ref{eq5:ratio1-esp}) that
$$
|\underline{H}(I)-H(k/N)|\le c_3 \min\Big\{\frac{\log\log N}{\log N}, \la(I)^{\rho_H}\Big\}, \quad\mbox{for all $k\in\underline{\nu_N} (I)$.}
$$
The latter inequality and the inclusion $\underline{\nu_N} (I)\subseteq \nu_N (I)$ imply that 
\begin{equation}
\label{eq6:ratio1-esp}
\frac{\sum_{k\in\underline{\nu_N} (I)} \big\|\Phi_\al (\cdot,H(k/N))\big\|_{L^\al(\R)}^\be\,|\underline{H}(I)-H(k/N)|\, N^{-\be H(k/N)}}{\sum_{k\in\nu_N (I)} \big\|\Phi_\al (\cdot,H(k/N))\big\|_{L^\al(\R)}^\be\,N^{-\be H(k/N)}}\le c_3 \min\Big\{\frac{\log\log N}{\log N}, \la(I)^{\rho_H}\Big\}.
\end{equation}
Next, observe that one has 
\begin{equation}
\label{eq7:ratio1-esp}
N^{-\be H(k/N)}\le N^{-\be \underline{H}(I)} (\log N)^{-4}, \quad\mbox{for all $k\in\overline{\nu_N} (I)$;}
\end{equation}
indeed, when $k\in\overline{\nu_N} (I)$, using (\ref{h:def1:HI}) and (\ref{eq5bis:ratio1-esp}), one gets that 
$$
N^{\be(\underline{H}(I)- H(k/N))}=N^{-\be|\underline{H}(I)- H(k/N)|}=e^{-\be |\underline{H}(I)- H(k/N)|\log N}
\le e^{-4\log\log N}=(\log N)^{-4}.
$$
Next, using (\ref{eq4:ratio1-esp}), (\ref{eq7:ratio1-esp}), (\ref{h:maxphial}), the inclusion $\overline{\nu_N} (I)\subseteq \nu_N (I)$ and (\ref{eq8:pratio1-esp}), one obtains that 
\begin{equation}
\label{eq8:ratio1-esp}
\frac{\sum_{k\in\overline{\nu_N} (I)} \big\|\Phi_\al (\cdot,H(k/N))\big\|_{L^\al(\R)}^\be\,|\underline{H}(I)-H(k/N)|\, N^{-\be H(k/N)}}{\sum_{k\in\nu_N (I)} \big\|\Phi_\al (\cdot,H(k/N))\big\|_{L^\al(\R)}^\be\,N^{-\be H(k/N)}}\le c_4 \,(\log N)^{-2}\,\la(I)^{\rho_H},
\end{equation}
where $c_4>0$ is a constant non depending on $I$ and~$N$. Next, putting together (\ref{eq3:ratio1-esp}), (\ref{eq5ter:ratio1-esp}), (\ref{eq6:ratio1-esp}),
(\ref{eq8:ratio1-esp}) and the inequality $\la(I)\le 1$, it follows that 
\begin{equation}
\label{eq9:ratio1-esp}
G_N \le c_5 \min\Big\{\frac{\log\log N}{\log N}, \la(I)^{\rho_H}\Big\},
\end{equation}
where the constant $c_5=c_1(c_3+c_4)$. Finally combining (\ref{eq2:ratio1-esp}) with (\ref{eq1:pratio1-esp}) and (\ref{eq9:ratio1-esp}), one can derive (\ref{eq1:ratio1-esp}).

Let us now prove that(\ref{eq2:ratioVNbis})holds. We set $c'=28/3$. Using (\ref{eq2:pratio1-esp}), (\ref{eq2june:pratio1-esp}), (\ref{h:def1:HI}), the fact that $\be\in (0,1/4]$, the fact that $H(\cdot)$ is with values $[\underline{H},\overline{H}]\subset (1/\al,1)$, and (\ref{card:nuI}), one gets that 
\[
\frac{\ESP \big(\wt{V}_{N}^{\be}(I)\big)}{2^{\be\underline{H}(I)}\,\ESP \big(\wt{V}_{2N}^{\be} (I)\big)}\le 2\frac{|\nu_{2N} (I)|}{|\nu_{N} (I)|}\le c',
\]
which shows that (\ref{eq1:ratioVNbis}) is satisfied. Next combining this inequality with the equality 
\[
\bigg |\frac{\ESP \big(\wt{V}_{N}^{\be}(I)\big)}{2^{\be\underline{H}(I)}\,\ESP \big(\wt{V}_{2N}^{\be} (I)\big)}-1\bigg |=\frac{\ESP \big(\wt{V}_{N}^{\be}(I)\big)}{2^{\be\underline{H}(I)}\,\ESP \big(\wt{V}_{2N}^{\be} (I)\big)}\bigg |\frac{2^{\be\underline{H}(I)}\,\ESP \big(\wt{V}_{2N}^{\be} (I)\big)}{\ESP \big(\wt{V}_{N}^{\be}(I)\big)}-1\bigg |
\]
and with (\ref{eq1:ratio1-esp}), we obtain (\ref{eq2:ratioVNbis}). 
\end{proof}

Now, we are in a position to prove Theorem~\ref{thm:main1}
\begin{proof}[Proof of Theorem~\ref{thm:main1}] Using (\ref{eq1:main1}), (\ref{h:def1:HI}), and (\ref{eq:troncC}), one gets that
$$
\be\,\big|\wh{H}_N^\be (I)-\min_{t\in I} H(t)\big|\le\Bigg |\log_2 \bigg(\frac{V_N^\be (I)}{V_{2N}^\be (I)}\bigg)-\log_2\big(2^{\be\underline{H}(I)}\big)\Bigg|=\Bigg|\log_2 \bigg (\frac{2^{\be \underline{H}(I)}\,V_{2N}^\be (I)}{V_N^\be (I)}\bigg)\Bigg|.
$$
Then, standard computations allow to derive that 
\begin{eqnarray}
\label{eq3:main1}
&&\be\,\big|\wh{H}_N^\be (I)-\min_{t\in I} H(t)\big|\\
&& \le \Bigg|\log_2\bigg(\frac{V_{2N}^\be (I)}{\ESP \big(\wt{V}_{2N}^{\be} (I)\big)}\bigg)\Bigg|+\Bigg|\log_2\bigg(\frac{V_{N}^\be (I)}{\ESP \big(\wt{V}_{N}^{\be} (I)\big)}\bigg)\Bigg|+\Bigg |\log_2\bigg(\frac{2^{\be\underline{H}(I)}\,\ESP \big(\wt{V}_{2N}^{\be} (I)\big)}{\ESP \big(\wt{V}_{N}^{\be}(I)\big)}\bigg)\Bigg |. \nonumber
\end{eqnarray}
Next observe that, for $j\in\{1,2\}$, the triangle inequality implies that 
\begin{equation}
\label{eq3bis:main1}
\bigg|\frac{V_{jN}^\be (I)}{\ESP \big(\wt{V}_{jN}^{\be} (I)\big)}-1\bigg|\le \bigg|\frac{\wt{V}_{j N}^\be (I)}{\ESP \big(\wt{V}_{jN}^{\be} (I)\big)}-1\bigg|+\frac{\big|V_{jN}^\be(I)-\wt{V}_{jN}^\be(I)\big|}{\ESP \big(\wt{V}_{jN}^{\be} (I)\big)}; 
\end{equation}
thus, assuming that $\ga>0$ is as in Lemma~\ref{lem:BorCan1}, it follows from (\ref{eq3:BorCan1}), (\ref{h:ecartdjktdjk}) and the first inequality in (\ref{eq2:encadEVT2}) that, one has, almost surely, for any $N\ge 4$ satisfying (\ref{eq1:encadEVT2}), 
\begin{equation}
\label{eq4:main1}
\bigg|\frac{V_{jN}^\be (I)}{\ESP \big(\wt{V}_{j N}^{\be} (I)\big)}-1\bigg|\le C_1 N^{-\ga} +C_2 N^{-\be(\rho_H-\underline{H}(I))}(\log N)^2,
\end{equation}
where $C_1$ and $C_2$ are two positive and finite random variables not depending on $N$ and $j$; also, we mention that $C_2$ does not depend on $I$, while $C_1$ depends on this interval. On the hand, observe that one has for some constant $c_3>0$,
\begin{equation}
\label{eq5:main1}
\big|\log_2 (1+x)\big|\le c_3 |x|, \quad \mbox{for all $x\in [-1/2,1/2]$.}
\end{equation}
Finally putting together, (\ref{eq3:main1}), (\ref{eq4:main1}), the inequality $\rho_H-\underline{H}(I)>0$, (\ref{eq1:ratio1-esp}) and (\ref{eq5:main1}), one can 
show that (\ref{eq2:main1}) holds.
\end{proof}

Let us now derive Theorem~\ref{thm:main1bis}

\begin{proof}[Proof of Theorem~\ref{thm:main1bis}] First observe that (\ref{eq1:main1}) and (\ref{eq:troncC}) imply that $\wh{H}_N^\be (I)\in [1/2,1]$. On the other hand,
one knows that $\underline{H}(I)\in (1/\al,1)\subset (1/2,1)$.  Therefore, one has 
\begin{equation}
\label{eq30:main1bis}
\big|\wh{H}_N^\be (I)-\underline{H}(I)\big|\le 1.
\end{equation}
The random variable $\big|\wh{H}_N^\be (I)-\underline{H}(I)\big|$ can also be bounded from above in a different way, namely
\begin{eqnarray}
\label{eq31:main1bis}
4^{-1}\be\big|\wh{H}_N^\be (I)-\underline{H}(I)\big| && \le \Big| \frac{V_N^{\be}(I)}{\ESP \big(\wt{V}_N^{\be} (I) \big)} -1  \Big|\times \Big (\frac{\ESP \big(\wt{V}_{N}^{\be}(I)\big)}{2^{\be\underline{H}(I)}\,\ESP \big(\wt{V}_{2N}^{\be} (I)\big)}\Big)\times \Big (\frac{\ESP \big(\wt{V}_{2N}^{\be} (I) \big)}{V_{2N}^{\be}(I)}\Big)\nonumber\\
&&\hspace{0.7cm}+\Big|\frac{\ESP \big(\wt{V}_{N}^{\be}(I)\big)}{2^{\be\underline{H}(I)}\,\ESP \big(\wt{V}_{2N}^{\be} (I)\big)}-1\Big|\times \Big (\frac{\ESP \big(\wt{V}_{2N}^{\be} (I) \big)}{V_{2N}^{\be}(I)}\Big)+\Big |\frac{\ESP \big(\wt{V}_{2N}^{\be} (I) \big)}{V_{2N}^{\be}(I)}-1\Big|\,.\nonumber\\
\end{eqnarray}
The inequality (\ref{eq31:main1bis}) can be derived as follows. In view of (\ref{eq1:main1}) and (\ref{eq:troncC}), one has 
\[
\be\big|\wh{H}_N^\be (I)-\underline{H}(I)\big|=\Bigg |\log_2 \Bigg (\Ch_\be\bigg(\frac{V_N^\be (I)}{V_{2N}^\be (I)}\bigg)\Bigg)-\log_2\Big ( \Ch_\be\big(2^{\be\underline{H}(I)}\big)\Big)\Bigg|\,. 
\]
Thus, applying the mean value theorem to the function $\log_2$, and using (\ref{eq:LiptroncC}) as well as the fact that $\Ch_\be (\cdot)$ is with values in $[2^{\be/2},2^\be]$, one gets that
\begin{equation}
\label{eq32:main1bis}
\be\big|\wh{H}_N^\be (I)-\underline{H}(I)\big|\le \frac{\be}{\log 2}\Big|\frac{V_N^\be (I)}{V_{2N}^\be (I)}-2^{\be\underline{H}(I)}\Big|\le 4\be \Big|\frac{V_N^\be (I)}{2^{\be\underline{H}(I)}V_{2N}^\be (I)}-1\Big|.
\end{equation}
Then (\ref{eq31:main1bis}) follows from (\ref{eq32:main1bis}), the fact that $\be\in (0,1/4]$, the equality
\[
\frac{V_N^\be (I)}{2^{\be\underline{H}(I)}V_{2N}^\be (I)}=\Big (\frac{V_N^{\be}(I)}{\ESP \big(\wt{V}_N^{\be} (I) \big)} \Big)\times \Big (\frac{\ESP \big(\wt{V}_{N}^{\be}(I)\big)}{2^{\be\underline{H}(I)}\,\ESP \big(\wt{V}_{2N}^{\be} (I)\big)}\Big)\times \Big (\frac{\ESP \big(\wt{V}_{2N}^{\be} (I) \big)}{V_{2N}^{\be}(I)}\Big),
\]
the equality $abc-1= (a-1)bc+(b-1)c+c-1$, for any $a,b,c\in\R$, and the triangle inequality.

Next we assume that the real number $\eta\in (0,1/2]$ is arbitrary and fixed. In order to obtain an convenient upper for the quantity $\ESP\big(\big|\wh{H}_N^\be (I)-\underline{H}(I)\big|^p\big)$, we will use (\ref{eq31:main1bis}) on a well chosen event $\ups_N (I,\eta)$, and on its complement $\overline{\ups}_N (I,\eta)=\O\setminus\ups_N (I,\eta)$ we will use (\ref{eq30:main1bis}). $\ups_N (I,\eta)$ is defined as 
\begin{equation}
\label{eq33:main1bis}
\ups_N (I,\eta)=\go_N (I,\eta) \cap\go_{2N} (I,\eta)\cap\ta_N, 
\end{equation}
where
\begin{eqnarray}
\label{eq3:main1bis}
\forall\, j\in\{1,2\},\quad \go_{jN} (I,\eta) &=& \bigg\{\o\in\O:\Big| \frac{\wt{V}_{jN}^{\be}(I,\o)}{\ESP \big(\wt{V}_{jN}^{\be} (I) \big)} -1  \Big| \le \eta\bigg\}\\
&=& \Big\{\o\in\O:(1-\eta)\ESP \big(\wt{V}_{jN}^{\be} (I) \big)\le \wt{V}_{jN}^{\be}(I,\o)\le (1+\eta)\ESP \big(\wt{V}_{jN}^{\be} (I) \big)\Big\},\nonumber
\end{eqnarray}
and 
\begin{equation}
\label{eq6:main1bis}
\ta_N=\big\{\o\in\O : C_* ^\be (\o)N^{-\be (\rho_H-\overline{H})}\le 4^{-1}\,\overline{c} \big\},
\end{equation}
$\overline{c}$ being the positive constant $c_2$ in (\ref{eq1:encadEVT}); we recall that the random variable $C_*$ has been introduced in Lemma~\ref{h:lem:ecartdjktdjk}, also, we recall that $\overline{H}=\sup_{t\in [0,1]} H(t)$. For $j\in\{1,2\}$, let $\overline{\go}_{jN} (I,\eta)=\O\setminus\go_{jN} (I,\eta)$ be the complement of $\go_{jN} (I,\eta)$. Observe that, it follows from (\ref{eq3:main1bis}), Lemma~\ref{lem:maj-prob1},  the inequality $\log (2N)\le 2\log(N)$, and the inequality $\underline{H}(I)<1$, that there is a constant $c_1>0$, not depending on $I$, $N$, $\eta$ and $\de$, such that for every $\de\in (0,1)$, one has
\begin{equation}
\label{eq40:main1bis}
\PR\big(\overline{\go}_N (I,\eta)\big)+\PR\big(\overline{\go}_{2N} (I,\eta)\big)\le c_1\,F_N \big(\eta,\la(I), 1-\de,\de\be (L-1)\big),
\end{equation}
where, for any $(\eta,\mu, x, y)\in (0,1/2]\times (0,1]\times (0,1)\times(0,+\infty)$,
\begin{equation}
\label{eq40bis:main1bis}
F_N(\eta,\mu, x, y) =\eta^{-1} N^{-y}\,(\log N)^{2}+\eta^{-4}\Big (\frac{N^{-2x}}{\mu^{2}}+\frac{N^{-4x}}{\mu^{4}}+N^{-4y}\Big)(\log N)^{8}\,.
\end{equation}
Let $\overline{\ta}_N=\O\setminus \ta_N$ be the complement of $\ta_N$. One can derive from (\ref{eq6:main1bis}) and (\ref{eq2:taildistA}) that, for some constant $c_2>0$, not depending on $I$, $N$, $\eta$ and $\de$, one has,
\begin{equation}
\label{eq41:main1bis}
\PR(\overline{\ta}_N)\le c_2\, N^{-\al(\rho_H-\overline{H})}\le c_2\,N^{-(\rho_H-\overline{H})},
\end{equation}
where the last inequality results from $\al>1$. Next, we set $c_3=c_1+c_2$, combining (\ref{eq30:main1bis}) with (\ref{eq33:main1bis}), (\ref{eq40:main1bis}) and (\ref{eq41:main1bis}), we obtain that
\begin{equation}
\label{eq42:main1bis}
\ESP\big(\big|\wh{H}_N^\be (I)-\underline{H}(I)\big|^p\,\one_{\overline{\ups}_N (I,\eta)}\big)\le \PR\big(\overline{\ups}_N (I,\eta)\big)\le c_3\Big (N^{-(\rho_H-\overline{H})}+F_N \big(\eta,\la(I), 1-\de,\de\be (L-1)\big)\Big)\,.
\end{equation}

Now, we look for an appropriate upper bound for $\ESP\big(\big|\wh{H}_N^\be (I)-\underline{H}(I)\big|^p\,\one_{\ups_N (I,\eta)}\big)$. To this end, we have to find, on the event $\ups_N (I,\eta)=\go_N (I,\eta) \cap\go_{2N} (I,\eta)\cap\ta_N$, a suitable upper bound for each term in the right-hand side of (\ref{eq31:main1bis}). So, let us derive the following three inequalities: 
\begin{equation}
\label{eq4:main1bis}
\forall\, j\in\{1,2\},\,\forall\,\o\in\go_{jN} (I,\eta), \quad \Big| \frac{V_{jN}^{\be}(I,\o)}{\ESP \big(\wt{V}_{jN}^{\be} (I) \big)} -1  \Big| \le \eta + C_*^\beta(\o) \,\overline{c} \, N^{-\be (\rho_H-\overline{H})},
\end{equation}

\begin{equation}
\label{eq7:main1bis}
\forall\,\o\in\go_{2N} (I,\eta)\cap \ta_N, \quad  \frac{\ESP \big(\wt{V}_{2N}^{\be} (I) \big)}{V_{2N}^{\be}(I,\o)}\le 4\,,
\end{equation}
and 
\begin{equation}
\label{eq8:main1bis}
\forall\,\o\in\go_{2N} (I,\eta)\cap \ta_N, \quad \bigg|\frac{\ESP \big(\wt{V}_{2N}^{\be} (I) \big)}{V_{2N}^{\be}(I)}-1\bigg|\le 4\Big (\eta+C_*^\beta (\o)\,\overline{c} \, N^{-\be (\rho_H-\overline{H})}\Big).
\end{equation}
Putting together the inequality
\[
\Big| \frac{V_{jN}^{\be}(I)}{\ESP \big(\wt{V}_{jN}^{\be} (I) \big)} -1  \Big| \le \Big| \frac{\wt{V}_{jN}^{\be}(I)}{\ESP \big(\wt{V}_{jN}^{\be} (I) \big)} -1  \Big| +\frac{\big|V_{jN}^{\be}(I)-\wt{V}_{jN}^{\be}(I)\big|}{\ESP \big(\wt{V}_{jN}^{\be} (I) \big)}, 
\]
the first equality in (\ref{eq3:main1bis}), (\ref{h:ecartdjktdjk}), and the first inequality in (\ref{eq1:encadEVT}), it follows that (\ref{eq4:main1bis}) holds. Now, we focus on the proof of (\ref{eq7:main1bis}).
%Notice that, since $\eta\in (0,1/2]$, a straightforward consequence of (\ref{eq5:main1bis}) is that 
%\begin{equation}
%\label{eq5:main1bis}
%\forall\,\o\in\go_N (I,\eta), \quad  \frac{V_N^{\be}(I,\o)}{\ESP \big(\wt{V}_N^{\be} (I) \big)}\le 2+C_*^\beta (\o)\,\overline{c} \, N^{-\be (\rho_H-\overline{H})}.
%\end{equation}
%The event $\go_{2N} (I,\eta)$ is defined through (\ref{eq3:main1bis}) in which $N$ is replaced by $2N$, and the event $\ta_N$ is defined as Let us show that
Using the triangle inequality, the second equality in (\ref{eq3:main1bis}), the inequality $1-\eta\ge 2^{-1}$, and (\ref{h:ecartdjktdjk}), we obtain that
\begin{eqnarray*}
V_{2N}^{\be}(I,\o) &\ge & \wt{V}_{2N}^{\be} (I,\o)-\big|V_{2N}^{\be}(I,\o)-\wt{V}_{2N}^{\be} (I,\o)\big| \\
&\ge & 2^{-1}\ESP \big(\wt{V}_{2N}^{\be} (I) \big)-C_* ^\be (\o)N^{-\be\rho_H}=\ESP \big(\wt{V}_{2N}^{\be} (I) \big)\bigg (2^{-1}-\frac{C_* ^\be (\o)N^{-\be\rho_H}}{\ESP \big(\wt{V}_{2N}^{\be} (I) \big)}\bigg).
\end{eqnarray*}
Then the first inequality in (\ref{eq1:encadEVT}) and (\ref{eq6:main1bis}) imply that 
\[
V_{2N}^{\be}(I,\o) \ge \ESP \big(\wt{V}_{2N}^{\be} (I) \big)\Big (2^{-1}-\overline{c}^{-1}\,C_* ^\be (\o)\, N^{-\be(\rho_H-\overline{H})}\Big)\ge 4^{-1}\,\ESP \big(\wt{V}_{2N}^{\be} (I) \big),
\] 
which proves that (\ref{eq7:main1bis}) is satisfied. Next, (\ref{eq8:main1bis}) can be obtained by combining the equality 
\[
\bigg|\frac{\ESP \big(\wt{V}_{2N}^{\be} (I) \big)}{V_{2N}^{\be}(I)}-1\bigg|=\frac{\ESP \big(\wt{V}_{2N}^{\be} (I) \big)}{V_{2N}^{\be}(I)}\bigg|\frac{V_{2N}^{\be}(I)}{\ESP \big(\wt{V}_{2N}^{\be} (I) \big)}-1\bigg|,
\]
 with (\ref{eq7:main1bis}) and (\ref{eq4:main1bis}), where $j=2$. Having proved the inequalities (\ref{eq4:main1bis}), (\ref{eq7:main1bis}), and (\ref{eq8:main1bis}), let us now notice that by combining them with (\ref{eq31:main1bis}),  (\ref{eq1:ratioVNbis}) and (\ref{eq2:ratioVNbis}), it follows that
\begin{equation}
\label{eq45:main1bis}
\big|\wh{H}_N^\be (I)-\underline{H}(I)\big| \,\one_{\ups_N (I,\eta)}\le c_4\Big (\eta+C_*^\be \,N^{-\be (\rho_H-\overline{H})}+T_N (\la (I), \be)\Big),
\end{equation}
where $c_4\ge 1$ is a constant not depending on $I$, $N$, $\eta$ and $\de$. Next, let $c_5$ be the constant defined as 
\[
c_5=(3c_4)^4 \, \ESP \Big(\big (1+C_*^\be \big)^4\Big);
\] 
observe that $c_5$ is finite since $4\be\le 1<\al$ (see Lemma~\ref{h:lem:ecartdjktdjk}). It follows from (\ref{eq45:main1bis}) that
\begin{equation}
\label{eq46:main1bis}
\ESP\big(\big|\wh{H}_N^\be (I)-\underline{H}(I)\big|^p\,\one_{\ups_N (I,\eta)}\big)\le c_5\Big (\eta^p+N^{-p\be (\rho_H-\overline{H})}+T_N (\la (I), \be)^p\Big).
\end{equation}
Next, setting $c_6=c_3+c_5$, then (\ref{eq42:main1bis}), (\ref{eq46:main1bis}), and the inequality $p\be\le 1$ entail that
\begin{equation}
\label{eq47:main1bis}
\ESP\big(\big|\wh{H}_N^\be (I)-\underline{H}(I)\big|^p\big)\le c_6\Big (F_N \big(\eta,\la(I), 1-\de,\de\be (L-1)\big)+\eta^p+N^{-p\be (\rho_H-\overline{H})}+T_N (\la (I), \be)^p\Big).
\end{equation}
Next, let $\Lam=\Lam(L,\be,p)$ be as in (\ref{eq2:main1bis}). We set
\begin{equation}
\label{eq48:main1bis}
\check{\de}=1-2^{-1}(p+4)\Lam
\end{equation}
and
\begin{equation}
\label{eq49:main1bis}
\check{\eta}_N=2^{-1}\la(I)^{-\Lam} N^{-\Lam} (\log N)^{2\Lam}.
\end{equation}
In view of (\ref{eq2:main1bis}) and (\ref{eq1:encadEVT2}), it can easily be seen that $\check{\de}\in (0,1)$ and $\check{\eta}_N\in (0,1/2]$. Also, observe that 
\begin{equation}
\label{eq50:main1bis}
p\Lam=2-2\check{\de}-4\Lam=\check{\de}\be (L-1)-\Lam.
\end{equation}
Standard computations, relying on (\ref{eq40bis:main1bis}), (\ref{eq0:ratio1-esp}),  (\ref{eq2:main1bis}), (\ref{eq1:encadEVT2}), (\ref{eq48:main1bis}), (\ref{eq49:main1bis}), (\ref{eq50:main1bis}), and the inequality $\la(I)\le 1$ allow to obtain, for some positive constant $c_7$ not depending on $N$, $I$, and $p$, the following inequality:
\begin{eqnarray}
\label{eq51:main1bis}
&&  c_7\Big (F_N \big(\check{\eta}_N,\la(I), 1-\check{\de},\check{\de}\be (L-1)\big)+(\check{\eta}_N)^p+N^{-p\be (\rho_H-\overline{H})}+T_N (\la (I), \be)^p\Big)\nonumber\\
&& \hspace{0.5cm} \le\min\bigg\{\Big (\frac{\log \log N}{\log N}\Big)^p ,\la(I)^{p\rho_H}\bigg\}+N^{-p\be (\rho_H-\overline{H})}\\
&&\hspace{1.5cm} +N^{-p \Lam}(\log N)^{8(1-\Lam)}\la(I)^{-4(1-\Lam )}\max\Big\{\la(I)^2, N^{-(p+4)\Lam }\Big\}.\nonumber
\end{eqnarray}
Finally, we can assume in (\ref{eq47:main1bis}) that $\eta=\check{\eta}_N$ and $\de=\check{\de}$, then combining (\ref{eq47:main1bis}) with the equality $\overline{H}=\sup_{t\in [0,1]} H(t)$ and with (\ref{eq51:main1bis}), we obtain (\ref{eq1:main1bis}).
\end{proof}

\section{Proofs of Theorems~\ref{thm:main2bis} and \ref{thm:main2}}
\label{sec:proofsmainbismain1}

As we have already pointed out (see Remark~\ref{rem:prthmain2bis}) for deriving Theorem~\ref{thm:main2bis} it is enough to show that  Proposition~\ref{lem:prmain2bis} holds.

\begin{proof}[Proof of Proposition~\ref{lem:prmain2bis}] We assume that the real numbers $\eta\in (0,1/2]$ and $\de\in (0,1)$ are arbitrary and fixed. Also, we assume that $N$ and $n$ are arbitrary integers satisfying $N\ge N_0$ and $0\le n <[\th_{N}^{-1}]$, where $N_0$ and $\th_N$ are as in Definition~\ref{def:main2}. The event $\ups_N (\I_{N,n},\eta)$ is defined through (\ref{eq33:main1bis}) with $I=\I_{N,n}$. The events $\Gamma_N (\eta)$ and $\overline{\Gamma}_N (\eta)$ are defined as 
\begin{equation}
\label{eq2:prmain2bis}
\Gamma_N (\eta)=\bigcap_{0\le n <[\th_{N}^{-1}]} \ups_N (\I_{N,n},\eta) \quad\mbox{and}\quad \overline{\Gamma}_N (\eta)=\O\setminus\Gamma_N (\eta). 
\end{equation}
Then, it follows from (\ref{eq45:main1bis}) that 
\begin{equation}
\label{eq3:prmain2bis}
 \one_{\Gamma_N (\eta)}\max_{0\le n <[\th_{N}^{-1}]} \big|\wh{H}_N^\be \big(\I_{N,n}\big(\I_{N,n}\big)-\underline{H}\big(\I_{N,n}\big)\big|\le c_1\Big (\eta+C_*^\be N^{-\be (\rho_H-\overline{H})}+T_N (\th_N, \be)\Big),
\end{equation}
where $c_1>0$ is a constant not depending on $N$ and $\eta$; recall that $\overline{H}=\sup_{t\in [0,1]} H(t)$ and that $T_N(\cdot,\cdot)$ has been introduced in (\ref{eq0:ratio1-esp}). Next, similarly to the proof (\ref{eq46:main1bis}), one can derive from (\ref{eq3:prmain2bis}), that
\begin{equation}
\label{eq4:prmain2bis}
\ESP\Big (\one_{\Gamma_N (\eta)}\max_{0\le n <[\th_{N}^{-1}]}\big|\wh{H}_N^\be \big(\I_{N,n}\big)-\underline{H}\big(\I_{N,n}\big)\big|^p\Big) \le c_2\Big (\eta^p+N^{-p\be (\rho_H-\overline{H})}+T_N (\th_N, \be)^p\Big),
\end{equation} 
where $c_2>0$ is a constant not depending on $N$, $\eta$, and $p$. On the other hand, similarly to (\ref{eq30:main1bis}), one has that
\[
\max_{0\le n <[\th_{N}^{-1}]} \big|\wh{H}_N^\be \big(\I_{N,n}\big)-\underline{H}\big(\I_{N,n}\big)\big|\le 1.
\]
Therefore, in view of (\ref{eq2:prmain2bis}), (\ref{eq33:main1bis}), (\ref{eq40:main1bis}), (\ref{eq40bis:main1bis}), and (\ref{eq41:main1bis}), we get that
\begin{eqnarray}
\label{eq5:prmain2bis}
&& \ESP\Big (\one_{\overline{\Gamma}_N (\eta)}\max_{0\le n <[\th_{N}^{-1}]} \big|\wh{H}_N^\be \big(\I_{N,n}\big)-\underline{H}\big(\I_{N,n}\big)\big|^p\Big)\nonumber\\
&& \le\PR(\overline{\ta}_N)+\sum_{n=0}^{[\th_{N}^{-1}]-1} \Big(\PR\big(\overline{\go}_N (\I_{N,n},\eta)\big)+\PR\big(\overline{\go}_{2N} (\I_{N,n},\eta)\big)\Big)\nonumber\\ 
&& \le c_3\Big (N^{-(\rho_H-\overline{H})}+\th_N ^{-1} F_N \big(\eta,\th_N, 1-\de,\de\be (L-1)\big)\Big),
\end{eqnarray}
where $c_3>0$ is a constant not depending on $N$, $\eta$, $\de$, and $p$. Next, putting together the second equality in (\ref{eq2:prmain2bis}), (\ref{eq4:prmain2bis}), (\ref{eq5:prmain2bis}), and the equality $p\be\le 1$, we obtain that 
\begin{eqnarray}
\label{eq6:prmain2bis}
&& \ESP\Big (\max_{0\le n <[\th_{N}^{-1}]}\big|\wh{H}_N^\be \big(\I_{N,n}\big)-\underline{H}\big(\I_{N,n}\big)\big|^p\Big) \\
&& \le c_4\Big (\th_N ^{-1} F_N \big(\eta,\th_N, 1-\de,\de\be (L-1)\big)+\eta^p+ N^{-p\be (\rho_H-\overline{H})}+T_N (\th_N, \be)^p\Big),\nonumber
\end{eqnarray}
where $c_4>0$ is a constant not depending on $N$, $\eta$, $\de$, and $p$.  Next, let $\Lam=\Lam(L,\be,p)$ and $\check{\de}\in (0,1)$ be as in (\ref{eq2:main1bis}) and (\ref{eq48:main1bis}). Moreover, we set $$\breve{\eta}_N=2^{-1}\th_N^{-\Lam} N^{-\Lam} (\log N)^{2\Lam};$$ notice that, in view of 
(\ref{eq1bis:main2}), one has $\breve{\eta}_N\in (0,1/2]$. Next, similarly to (\ref{eq51:main1bis}), we can obtain, for some constant $c_5>0$, not depending on $N$ and $p$, the following inequality:
\begin{eqnarray}
\label{eq7:prmain2bis}
&&  c_5\Big (\th_N ^{-1}F_N \big(\breve{\eta}_N,\th_N, 1-\check{\de},\check{\de}\be (L-1)\big)+(\breve{\eta}_N)^p+N^{-p\be (\rho_H-\overline{H})}+T_N (\th_N, \be)^p\Big)\nonumber\\
&& \hspace{0.5cm} \le\min\bigg\{\Big (\frac{\log \log N}{\log N}\Big)^p ,(\th_N)^{p\rho_H}\bigg\}+N^{-p\be (\rho_H-\overline{H})}\\
&&\hspace{1.5cm} +N^{-p \Lam}(\log N)^{8(1-\Lam)}\th_N ^{-4(\frac{5}{4}-\Lam )}\max\Big\{\th_N^2, N^{-(p+4)\Lam }\Big\}\,.\nonumber
\end{eqnarray}
Finally, we can assume in (\ref{eq6:prmain2bis}) that $\eta=\breve{\eta}_N$ and $\de=\check{\de}$, then combining (\ref{eq6:prmain2bis}) with the equality $\overline{H}=\sup_{t\in [0,1]} H(t)$ and with (\ref{eq7:prmain2bis}), we obtain (\ref{eq1:prmain2bis}).
\end{proof}

As we have already pointed out (see Remark~\ref{rem:prthmain2}) for deriving Theorem~\ref{thm:main2} it is enough to show that  Proposition~\ref{lem:maxhaHhtH} holds. In order to prove this proposition, we need the following lemma which is reminiscent of Lemma~\ref{lem:BorCan1} and Remark~\ref{Rem:BorCan1}.

\begin{Lem}
\label{lem:BorCan2}
Assume that the real number $\be\in (0,1/4]$ and the integer $L$ are arbitrary and satisfy (\ref{eq1:BorCan1}). Also, assume that (\ref{eq2:main2}) holds. Then, using some of the notations introduced in Definition~\ref{def:main2}, there exists a positive finite random variable $C$, such that the inequality 
\begin{equation}
\label{eq1:BorCan2}
\forall\,j\in\{1,2\},\quad\max_{0\le n <[\th_{N}^{-1}]}\bigg| \frac{\wt{V}_{jN}^{\be}(\I_{N,n})}{\ESP \big(\wt{V}_{jN}^{\be} (\I_{N,n}) \big)} -1  \bigg| \le C \th_N ^{-1} N^{-\frac{\be(L-1)-2}{4\be(L-1)+2}},
\end{equation}
holds almost surely, for any integer $N\ge N_0$.
\end{Lem}

\begin{proof}[Proof of Lemma~\ref{lem:BorCan2}]  Let us set 
\begin{equation}
\label{eq1bis:BorCan2}
\ga_0=\frac{\be(L-1)-2}{4\be(L-1)+2}\quad\mbox{and}\quad \de_0=\frac{5}{4\be(L-\overline{H})+2};
\end{equation} 
in view of  (\ref{eq1:BorCan1}) and the fact that $\overline{H}=\sup_{t\in [0,1]} H(t)<1$, standard computations allow to show that
\begin{equation}
\label{eq1ter:BorCan2}
0<\de_0<2^{-1}\quad\mbox{and}\quad\de_0\be(L-\overline{H})-\ga_0 >1\quad\mbox{and}\quad 2(1-\de_0-2\ga_0)>1.
\end{equation}
Next, observe that
\begin{eqnarray}
\label{eq2:BorCan2}
&& \PR\Bigg( \max_{0\le n <[\th_{N}^{-1}]}\bigg| \frac{\wt{V}_{jN}^{\be}(\I_{N,n})}{\ESP \big(\wt{V}_{jN}^{\be} (\I_{N,n}) \big)} -1  \bigg|>\th_N ^{-1} N^{-\ga_0}\Bigg)\\
&& \hspace{5cm}\le \sum_{n=0}^{[\th_{N}^{-1}]-1}\PR\Bigg( \bigg| \frac{\wt{V}_{jN}^{\be}(\I_{N,n})}{\ESP \big(\wt{V}_{jN}^{\be} (\I_{N,n}) \big)} -1  \bigg|>\th_N ^{-1} N^{-\ga_0}\Bigg).\nonumber 
\end{eqnarray}
Moreover, we know from Definition~\ref{def:main2} that $4(L+1)\,(jN)^{-1}\big(\log (jN)\big)^2\le \th_N \le \la(\I_{N,n})$, the integers $j\in\{1,2\}$, $N\ge N_0$, and $n\in \big\{0,\ldots, [\th_N ^{-1}]-1\big\}$ being arbitrary. Thus, Lemma~\ref{lem:maj-prob1} (in which $I=\I_{N,n}$, $\de=\de_0$, $\eta=\th_N ^{-1} N^{-\ga_0}$, and $\wt{V}_{N}^{\be}(I)$ is replaced by $\wt{V}_{jN}^{\be}(\I_{N,n})$)
can be used in order to bound from above the probabilities in the right-hand side of  (\ref{eq2:BorCan2}). We obtain, in this way,
\begin{eqnarray}
\label{eq3:BorCan2}
&& \PR\Bigg( \Big| \frac{\wt{V}_{jN}^{\be}(\I_{N,n})}{\ESP \big(\wt{V}_{jN}^{\be} (\I_{N,n}) \big)} -1  \Big| >\th_N ^{-1} N^{-\ga_0} \Bigg)\nonumber\\
&& \le c_1\,\th_N N^{\ga_0} (jN)^{-\de_0\be (L-\underline{H}(\I_{N,n}))}\big(\log(j N)\big)^{2}\nonumber\\
&& \hspace{1cm}+c_1\, \th_N^4 N^{4\ga_0}\bigg (\frac{(jN)^{-2(1-\de_0)}}{\th_N^{2}}+\frac{(jN)^{-4(1-\de_0)}}{\th_N^{4}}+(jN)^{-4\de_0\be (L-\underline{H}(\I_{N,n}))}\bigg)\big(\log(j N)\big)^{8}\nonumber\\
&& \le  c_2\, \th_N N^{\ga_0-\de_0\be (L-\underline{H}(\I_{N,n}))}(\log N)^{2}\nonumber\\
&& \hspace{1cm}+c_2\, \th_N \bigg ( \th_N N^{-2(1-\de_0-2\ga_0)}+\th_N ^{-1} N^{-4(1-\de_0-\ga_0)}+\th_N ^3 N^{4\ga_0-4\de_0\be (L-\underline{H}(\I_{N,n}))}\bigg)(\log N)^{8}\nonumber\\
&& \le  c_2\, \th_N N^{\ga_0-\de_0\be (L-\overline{H})}(\log N)^{2}\nonumber\\
&& \hspace{1cm}+c_2\, \th_N \bigg (N^{-2(1-\de_0-2\ga_0)}+\th_N ^{-1} N^{-4(1-\de_0-\ga_0)}+N^{4\ga_0-4\de_0\be (L-\overline{H})}\bigg)(\log N)^{8},
\end{eqnarray}
where $c_1>0$ and $c_2=8c_1$ are two constants not depending on $N$, $n$, $\ga_0$ and $\de_0$ (in fact $c_1$ is the constant $c$ in (\ref{eq1:maj-prob1})). Next, combining (\ref{eq2:BorCan2}) with (\ref{eq3:BorCan2}), we get that 
that 
\begin{eqnarray*}
&& \PR\Bigg( \max_{0\le n <[\th_{N}^{-1}]}\bigg| \frac{\wt{V}_{jN}^{\be}(\I_{N,n})}{\ESP \big(\wt{V}_{jN}^{\be} (\I_{N,n}) \big)} -1  \bigg|>\th_N ^{-1} N^{-\ga_0}\Bigg)\\
&& \le  c_2\,  N^{\ga_0-\de_0\be (L-\overline{H})}(\log N)^{2}\\
&& \hspace{1cm}+c_2\,  \bigg (N^{-2(1-\de_0-2\ga_0)}+\th_N ^{-1} N^{-4(1-\de_0-\ga_0)}+N^{4\ga_0-4\de_0\be (L-\overline{H})}\bigg)(\log N)^{8}.
\end{eqnarray*}
Finally, putting together the last inequality, (\ref{eq1ter:BorCan2}), the first equality in (\ref{eq1bis:BorCan2}), and the second equality in (\ref{eq2:main2}), it follows that
$$
\sum_{N=N_0}^{+\infty}  \PR\Bigg( \max_{0\le n <[\th_{N}^{-1}]}\bigg| \frac{\wt{V}_{jN}^{\be}(\I_{N,n})}{\ESP \big(\wt{V}_{jN}^{\be} (\I_{N,n}) \big)} -1  \bigg|>\th_N ^{-1} N^{-\ga_0}\Bigg)<+\infty,
$$
then using Borel-Cantelli Lemma, we obtain (\ref{eq1:BorCan2}).
\end{proof}

Now, we are in a position to prove Proposition~\ref{lem:maxhaHhtH}.

\begin{proof}[Proof of Proposition~\ref{lem:maxhaHhtH}] Similarly to (\ref{eq3:main1}) and (\ref{eq3bis:main1}), we have
\begin{eqnarray}
\label{eq2:maxhaHhtH}
&&\be\max_{0\le n <[\th_{N}^{-1}]} \big|\wh{H}_N^\be (\I_{N,n})-\underline{H}\big(\I_{N,n}\big)\big|\nonumber\\
&& \le \max_{0\le n <[\th_{N}^{-1}]} \Bigg|\log_2\bigg(\frac{V_{2N}^\be (\I_{N,n})}{\ESP \big(\wt{V}_{2N}^{\be} (\I_{N,n})\big)}\bigg)\Bigg|\\
&& \hspace{1cm}+\max_{0\le n <[\th_{N}^{-1}]} \Bigg|\log_2\bigg(\frac{V_{N}^\be (\I_{N,n})}{\ESP \big(\wt{V}_{N}^{\be} (\I_{N,n})\big)}\bigg)\Bigg|+\max_{0\le n <[\th_{N}^{-1}]} \Bigg |\log_2\bigg(\frac{2^{\be\underline{H}(\I_{N,n})}\,\ESP \big(\wt{V}_{2N}^{\be} (\I_{N,n})\big)}{\ESP \big(\wt{V}_{N}^{\be}(\I_{N,n})\big)}\bigg)\Bigg |. \nonumber
\end{eqnarray}
and, for $j\in\{1,2\}$,
\begin{eqnarray}
\label{eq3:maxhaHhtH}
\max_{0\le n <[\th_{N}^{-1}]} \bigg|\frac{V_{j N}^\be (\I_{N,n})}{\ESP \big(\wt{V}_{jN}^{\be} (\I_{N,n})\big)}-1\bigg| &\le & \max_{0\le n <[\th_{N}^{-1}]}\bigg|\frac{\wt{V}_{jN}^\be (\I_{N,n})}{\ESP \big(\wt{V}_{jN}^{\be} (\I_{N,n})\big)}-1\bigg|\\
&&\hspace{1cm}+\max_{0\le n <[\th_{N}^{-1}]}\frac{\big|V_{jN}^\be(\I_{N,n})-\wt{V}_{jN}^\be(\I_{N,n})\big|}{\ESP \big(\wt{V}_{jN}^{\be} (\I_{N,n})\big)}\,.\nonumber
\end{eqnarray}
Next, we denote by $C_1$ the random variable $C$ in Lemma~\ref{lem:BorCan2}; recall that this random variable does not depend on $N$. On the other hand, we assume that the random variable $C_2$ is defined as $C_2=(\overline{c})^{-1}\,C_* ^\be$, where $\overline{c}$ is the positive constant $c_2$ in (\ref{eq1:encadEVT}), and the random variable $C_*$ is as in Lemma~\ref{h:lem:ecartdjktdjk}. Observe that we know from Remark~\ref{rem:encadEVT} and Lemma~\ref{h:lem:ecartdjktdjk} that $C_2$ does not depend on $N$ and on the intervals $\I_{N,n}$. It follows from (\ref{eq3:maxhaHhtH}), (\ref{eq1:BorCan2}), (\ref{h:ecartdjktdjk}), and the first inequality in (\ref{eq1:encadEVT}) that, we have, almost surely, for any $N\ge N_0 $, 
\begin{equation}
\label{eq4:maxhaHhtH}
\max_{0\le n <[\th_{N}^{-1}]} \bigg|\frac{V_{N}^\be (\I_{N,n})}{\ESP \big(\wt{V}_{N}^{\be} (\I_{N,n})\big)}-1\bigg|\le C_1 \th_N ^{-1} N^{-\frac{\be(L-1)-2}{4\be(L-1)+2}} +C_2 N^{-\be(\rho_H-\sup_{t\in [0,1]} H(t))}\,.
\end{equation}
On the other hand, we know from Lemma~\ref{lem:ratio1-esp} and from Definition~\ref{def:main2} that
\begin{eqnarray}
\label{eq5:maxhaHhtH}
\max_{0\le n <[\th_{N}^{-1}]} \bigg |\frac{2^{\be\underline{H}(\I_{N,n})}\,\ESP \big(\wt{V}_{2N}^{\be} (\I_{N,n})\big)}{\ESP \big(\wt{V}_{N}^{\be}(\I_{N,n})\big)}\bigg |\ &\le &  c_3 \min\Big\{\frac{\log \log N}{\log N}, \th_N^{\rho_H}\Big\}\\
&& \hspace{0.5cm}+c_3 N^{-\be\rho_H}(\log N)^3+c_3\, \th_N^{-1}N^{-1} (\log N)^2,\nonumber
\end{eqnarray}
where $c_3>0$ is a constant not depending on $N$. Finally, putting together (\ref{eq2:maxhaHhtH}), (\ref{eq4:maxhaHhtH}), (\ref{eq5:maxhaHhtH}), and (\ref{eq5:main1}), we obtain the proposition.
\end{proof}

\vspace{0.5cm}
\noindent{\bf Acknowledgement.} The author authors are very grateful to the anonymous associate editor and two referees for their valuable comments and suggestions which have led to improvements of the article. This work has been partially supported by ANR-11-BS01-0011 (AMATIS), GDR 3475 (Analyse Multifractale), and ANR-11-LABX-0007-01 (CEMPI). 	
 	
%\newpage
%\selectlanguage{english}
\thispagestyle{plain}
\bibliographystyle{amsplain}
%\nocite{*}
\bibliography{mabibliodethese3}
\end{document}